\documentclass{amsart}
\usepackage{amsmath,amsthm}
\usepackage{amsfonts,amssymb}
\usepackage{accents}
\usepackage{enumerate}
\usepackage{accents,color}
\usepackage{graphicx}
\usepackage{pict2e}
\newcommand{\lvt}{\left|\kern-1.35pt\left|\kern-1.3pt\left|}
\newcommand{\rvt}{\right|\kern-1.3pt\right|\kern-1.35pt\right|}

\newtheorem{thm}{Theorem}[section]
\newtheorem{cor}[thm]{Corollary}
\newtheorem{lem}[thm]{Lemma}
\newtheorem{prop}[thm]{Proposition}

\newtheorem{defn}[thm]{Definition}

\theoremstyle{remark}

 \def\la{{\langle}}
 \def\ra{{\rangle}}

 \def\d{\mathrm{d}}
 
 \def\sph{{\mathbb{S}^{d-1}}}
 
 \def\sa{{\mathsf a}}
 \def\sb{{\mathsf b}}
 
 \def\sd{{\mathsf d}}
  
 \def\sh{{\mathsf h}}
  
 \def\sw{{\mathsf w}}

 \def\sE{{\mathsf E}}
 
 \def\sH{{\mathsf H}}

 \def\sO{{\mathsf O}}
 \def\sP{{\mathsf P}}
 \def\sQ{{\mathsf Q}}
 \def\sR{{\mathsf R}}
 \def\sS{{\mathsf S}}
 \def\sT{{\mathsf T}}
 
 \def\sW{{\mathsf W}}
 
 \def\sY{{\mathsf Y}}
 
 \def\fa{{\mathfrak a}}
 \def\fb{{\mathfrak b}}

 \def\a{{\alpha}}
 \def\b{{\beta}}
 \def\g{{\gamma}}
 
 \def\t{{\theta}}
 \def\l{{\lambda}}
 \def\o{{\omega}}
 \def\s{\sigma}
 \def\la{{\langle}}
 \def\ra{{\rangle}}

 \def\kb{{\mathbf k}}

 \def\Hb{{\mathbf H}}
 \def\Jb{{\mathbf J}}
 
 \def\Pb{{\mathbf P}}
 \def\Qb{{\mathbf Q}}
 
 \def\Sb{{\mathbf S}}
 
 \def\Wb{{\mathbf W}}
 
 \def\CD{{\mathcal D}}
 
 \def\CH{{\mathcal H}}

 \def\CV{{\mathcal V}}

 \def\BB{{\mathbb B}}

 \def\NN{{\mathbb N}}

 \def\RR{{\mathbb R}}
 \def\SS{{\mathbb S}}

 \def\VV{{\mathbb V}}
 
 \def\XX{{\mathbb X}}
 
      \def\proj{\operatorname{proj}}

\def\lla{\langle{\kern-2.5pt}\langle}      
\def\rra{\rangle{\kern-2.5pt}\rangle}

\newcommand{\wh}{\widehat}

\def\f{\frac}

\graphicspath{{figures/}}
\begin{document}
 
\title{Orthogonal polynomials on domains of revolution}
\author{Yuan~Xu}
\address{Department of Mathematics, University of Oregon, Eugene,
OR 97403--1222, USA}
\email{yuan@uoregon.edu}
\thanks{The author was partially supported by Simons Foundation Grant \#849676}
\date{\today}

\begin{abstract}
We study orthogonal polynomials for a weight function defined over a domain of revolution, where the domain is formed 
from rotating a two-dimensional region and goes beyond the quadratic domains. Explicit constructions of orthogonal bases are 
provided for weight functions on a number of domains. Particular attention is paid to the setting when an orthogonal 
basis can be constructed explicitly in terms of known polynomials of either one or two variables. Several new families of 
orthogonal polynomials are derived, including a few families that are eigenfunctions of a spectral operator and their 
reproducing kernels satisfy an addition formula. 
\end{abstract}

\maketitle

\section{Introduction}
\setcounter{equation}{0}

The structure of orthogonal polynomials (OPs) in the multidimensional setting is much more complex than that in one variable. 
Not only an orthogonal basis is more difficult to compute and work with (cf. \cite{BPS, DX, LN, OST}), but it is also much harder
to utilize multivariate OPs in analysis and computational analysis. It is therefore not surprising that the most well-studied OPs 
are those families on regular domains, for which an orthogonal basis can be given explicitly in terms of the OPs in one 
variable (cf. \cite{DX}). Such domains include, for example, the product domains, the balls, and simplexes in the Euclidean space. 
The latter two cases have classical OPs that are orthogonal with respect to weight functions akin to the Jacobi weight in one 
variable, and they have been extensively studied in the literature, not only for computational purposes but also as main tools 
for approximation theory and harmonic analysis on the domains (cf. \cite{DaiX, DX}). In both cases, an orthogonal basis can 
be given in terms of the Jacobi polynomials, which makes it possible to develop fast algorithms for computing the bases 
(cf. \cite{OST}), and to study the orthogonal structure and uncover its intrinsic properties. It follows, for example, that the space 
of OPs of degree $n$ is an eigenspace of a second-order differential operator, aka {\it spectral operator}, and the reproducing 
kernel of the space enjoys a closed-form formula, aka {\it addition formula}. Together, the spectral operator and the addition 
formula provide essential tools for approximation theory and harmonic analysis in the localizable space of homogeneous 
type \cite{DaiX, X21, X23a}.

Recently, new families of OPs have been studied for quadratic domains of revolution. There are two
types of such domains, quadratic surfaces, and solid domains bounded by quadratic surfaces and, if compact, hyperplanes. 
The latter is of the form
$$
\VV^{d+1} = \left \{(x,t): \|x\| \le \phi(t), \quad x \in \RR^d, \,\, t \in [a,b] \right \}
$$ 
where $[a,b]$ could be $\RR_+$ or $\RR$ and $\phi$ is either a linear polynomial or the square root of 
a polynomial of degree at most $2$, nonnegative over $[a,b]$. The domain includes cones, hyperboloids, and 
paraboloids in $\RR^{d+1}$ \cite{OX1, X20, X21a, X23a}. In several of these domains, orthogonal bases for a family of 
weight functions can be given explicitly in terms of OPs of one variable, which greatly facilitates their computation. 
More importantly, some of the bases can be given in terms of the classical OPs of one variable, which makes it 
possible to identify domains and weight functions that possess a spectral operator and an addition formula. In particular, 
we know that both properties exist for a family of weight functions on the cone \cite{X20} and, at least partially,
on the hyperboloid \cite{X21a}. These properties allow us to carry out extensive analysis for several problems on these 
domains \cite{X21, X21b, X23a, X23b}. 

The purpose of this paper is to study OPs on other solid domains, beyond the quadratic ones, of revolution. In the 
case of the quadratic domain of revolution, our OPs are constructed as a wrapped product of polynomials in the $t$
variable and OPs on the unit ball $\BB^d$ of $\RR^d$, and their orthogonality is established via the decomposition 
of the integral
$$
  \int_{\VV^{d+1}} f(x,t) w(t) (t^2-\|x\|^2)^\mu \d x \d t = \int_0^\infty |\phi(t)|^d \int_{\BB^d} f(t y,t) (1-\|y\|^2)^\mu \d y\, w(t) t^{2 \mu} \d t.
$$
For the study in this paper, we change our point of view and obtain a domain in $\RR^{3}$, 
for example, by rotating a domain in the positive quadrant of $\RR^2$ around the $t$ axis. More generally, let 
$\Omega_+$ be a domain in $\RR_+^2 = \{(s,t): t \ge 0\}$ and symmetric in the $s$ variable. We consider the domain 
of revolution defined by
$$
   \VV^{d+1} = \left\{(x,t) \in \RR^{d+1}:  (\|x\|, |t|) \in \Omega_+, \,\, x \in \RR^d, \,\, t\in \RR \right\}, 
$$
which is not necessarily a quadratic domain, and equipped it with the weight function $\sw(\|x\|,t)$, where $\sw$ is 
defined on $\Omega_+$ and is symmetric in its first variable. In this setting, the integral over the domain can be 
decomposed as 
$$
   \int_{\VV^{d+1}} f(x,t) w(\|x\|,t) \d x \d t  = \int_{\Omega_+} \int_{\sph} f(s \xi,t) \d\s(\xi) \sw(s,t) s^{d-1} \d s \d t,
$$
which suggests that one may construct OPs on $\VV^{d+1}$ by OPs of two variables on $\Omega_+$ and spherical
harmonics. The construction, however, is not obvious, nor is it simple, especially if we want to obtain bases given 
explicitly in terms of classical OPs. What it entails is a careful study of the relation between OPs on $\Omega_+$ 
and those on $\VV^{d+1}$. 

We shall concentrate, as reasoned in the first paragraph, on orthogonal bases that can be constructed in terms of OPs
that are explicitly known, either as OPs of two variables or one variable. While our study leads to several new families of
OPs for domains, and weight functions, that have not been considered before, the study also shows that the construction 
of orthogonal bases on $\VV^{d+1}$ can be quite subtle. For example, the cone can be regarded as a rotation with 
$\Omega_+$ being the right triangle. One may ask if we can consider other triangles, which are after all just an affine
mapping from the right triangle. The answer, however, turns out to be mostly negative for the cone but, nevertheless, 
positive for the hyperboloid and double cone. If the domain $\Omega_+$ is extended symmetrically to the negative 
values of $t$ variable so that $\VV^{d+1}$ is symmetric in the $t$ direction, we obtain several distinct domains with 
weight functions that possess not only explicit orthogonal bases but also the spectral operator and addition 
formula, and they can be derived by relating to the results of the hyperboloid. These results expand our knowledge of 
domains that possess these two essential properties and, as a result, grant access to the framework \cite{X21} for
approximation and harmonic analysis on such domains. 

The paper is organized as follows. The next section is devoted to the background and preliminary and contains 
a review of several families of classical OPs that will be needed. The new setup for OPs over domains of revolution will be 
developed in the third section, where no additional restriction is imposed on $\Omega_+$. In the fourth section, we 
assume that the domain $\Omega_+$ can be extended to a fully symmetric domain $\Omega$ in $\RR^2$ and study 
OPs of two variables on $\Omega$, so that they can be used for constructing orthogonal bases on domains of revolution. 
While several examples based on parallelograms in $\RR^2$ are given in the fourth section, the examples corresponding
to triangle domains are given in the fifth section, which contains new families of OPs that possess the two 
essential properties mentioned above. 

\section{Background and preliminary}
\setcounter{equation}{0}

We start with a review of classical OPs of one variable and two variables, which will be needed later, in the first two 
subsections, and spherical harmonics and classical OPs on the unit ball in the third section. In the fourth subsection, 
we recall what is known for OPs over domains of revolution and lay down the basics for the study in the latter sections. 

\subsection{OPs of one variable} 
We are interested in orthogonal bases that can be expressed in terms of classical OPs. Since we are mostly interested
in the compact setting, we first recall the Jacobi polynomials and their variants.

\subsubsection{Jacobi polynomials} For $\a, \b > -1$, the Jacobi weight function is defined by 
$$
      w_{\a,\b}(t):=(1-t)^\a(1+t)^\b, \qquad -1 < x <1. 
$$
Its normalization constant $c'_{\a,\b}$, defined by $c'_{\a,\b}  \int_{-1}^1 w_{\a,\b} (x)dx = 1$, is given by
\begin{equation}\label{eq:c_ab}
  c'_{\a,\b} = \frac{1}{2^{\a+\b+1}} c_{\a,\b} \quad\hbox{with} \quad 
   c_{\a,\b} := \frac{\Gamma(\a+\b+2)}{\Gamma(\a+1)\Gamma(\b+1)},
\end{equation}
where $c_{\a,\b}$ is the normalization constant of the Jacobi weight $t^\a(1-t)^\b$ on the interval $[0,1]$. 
The Jacobi polynomial of degree $n$ is defined by 
$$
  P_n^{(\a,\b)}(t) = \frac{(\a+1)_n}{n!} {}_2F_1 \left (\begin{matrix} -n, n+\a+\b+1 \\
      \a+1 \end{matrix}; \frac{1-t}{2} \right).
$$
These polynomials are orthogonal with respect to $w_{\a,\b}$ on $[-1,1]$; more precisely, 
$$
c_{\a,\b}' \int_{-1}^1 P_n^{(\a,\b)}(t) P_m^{(\a,\b)}(t) w_{\a,\b}(t) \d t = h_n^{(\a,\b)} \delta_{m,n},
$$
where $h_n^{(\a,\b)}$ is the square of the $L^2$ norm that satisfies (see, for example, \cite[p. 21]{DX})
\begin{equation}\label{eq:hn_ab}
  h_n^{(\a,\b)} =  \frac{(\a+1)_n (\b+1)_n(\a+\b+n+1)}{n!(\a+\b+2)_n(\a+\b+2 n+1)}.
\end{equation}
For $\a = \b = \l -\frac12$, the weight function is the Gegenbauer weight
$$
w_\l(t):= w_{\l- \f12, \l - \f12}(t) = (1-t^2)^{\l-\f12}, \qquad \l >  - \tfrac12,
$$
and the corresponding OPs are the Gegenbuer polynomials $C_n^\l$, usually normalized by 
$C_n^\l(1) = \frac{(2\l)_n}{n!}$, where $(a)_n = a (a+1)\cdots (a+n-1)$ is the Pochhammer symbol.

\medskip\noindent
\subsubsection{Generalized Gegenbauer polynomial} For $\l,\mu > -\f12$, the generalized Gegenbauer polynomials 
$C_n^{(\l,\mu)}$ satisfy the orthogonal relation
$$
     \frac{\Gamma(\l+\mu)}{\Gamma(\l+\f12)\Gamma(\mu+\f12)} \int_{-1}^1 C_n^{(\l,\mu)}(x) C_m^{(\l,\mu)}(y) 
        |x|^{2\mu} (1-x^2)^{\l -\f12} \d x  =  \sh_n^{(\l,\mu)} \delta_{m,n}.  
$$
The polynomials $C_n^{(\l,\mu)}$ are given explicitly by \cite[Section 1.5.2]{DX}  
\begin{align} \label{eq:gGegen}
\begin{split}
C_{2m}^{(\l,\mu)}(t) &\, = \frac{(\l+\mu)_m}{(\mu+\f12)_m} P_m^{(\l-\f12,\mu-\f12)}(2 t^2 -1),\\
C_{2m+1}^{(\l,\mu)}(t) &\, = \frac{(\l+\mu)_{m+1}}{(\mu+\f12)_{m+1}} t P_m^{(\l-\f12,\mu+\f12)}(2 t^2 -1),
\end{split}
\end{align}
where $P_n^{(a,b)}$ are the standard Jacobi polynomials. The norm square of these polynomials
are equal to \cite[p. 26]{DX}
\begin{align} \label{eq:gGegenNorm}
\begin{split}
\sh_{2m}^{(\l,\mu)} &\, = \frac{(\l+\f12)_m(\l+\mu)_m(\l+\mu)}{m!(\mu+\f12)_m(\l+\mu+2m)},\\  
 \sh_{2m+1}^{(\l,\mu)}&\,  =  \frac{(\l+\f12)_m(\l+\mu)_{m+1}(\l+\mu)}{m!(\mu+\f12)_{m+1}(\l+\mu+2m+1)}. 
\end{split}
\end{align}
Furthermore, in terms of the evaluation of the polynomials at $t=1$, we can write
\begin{equation} \label{eq:gGegen@1}
   \sh_{n}^{(\l,\mu)} = \frac{\l+\mu}{n+\l+\mu} C_{n}^{(\l,\mu)}(1). 
\end{equation}

\subsection{Orthogonal polynomials in two variables} 

Let $\Omega$ be a compact domain in $\RR^2$ with a positive area. Let $\sW$ be a weight function on
$\Omega$.\footnote{Throughout the paper we adopt the convention that the letters $\sP, \sQ, \sS, \sW$ etc., 
in the sansmath font, are reserved for functions and polynomials in two variables.} 
We consider OPs with respect to the inner product
$$
  \la f, g\ra_\sW = c_\sW \int_\Omega f(u,v) g(u,v) \sW(u,v) \d u \d v,
$$
where $c_\sW$ is the normalization constant of $\sW$ such that $\la 1, 1 \ra_\sW = 1$. We denote by 
$\CV_n(\sW, \Omega)$ the space of OPs of degree $n$ for $n \in \NN_0$. It is known that 
$$
\dim \CV_n(\sW,\Omega) = n+1, \qquad n = 0,1, 2, \ldots. 
$$
We give two examples that will serve as the building blocks for our orthogonal basis on higher dimensions. 

\subsubsection{The product domain $\square= [-1,1]^2$}
Let $\sW(u,v) = w_1(u) w_2(v)$, where $w_1$ and $w_2$ are weight functions on $[-1,1]$. Let $p_n(w_j)$ be orthogonal
polynomials of degree $n$ with respect to the weight function $w_j$ for $j =1,2$. Then the product polynomials 
$$
  \sP_k^n(u,v) = p_k(w_1; x) p_{n-k}(w_2; y), \qquad 0 \le k \le n,
$$
form an orthogonal basis for $\CV_n(\sW,\square)$. 

\subsubsection{The triangle $\triangle$ and the Jacobi polynomials on the triangle}
Let $\triangle$ denote the triangle defined by 
$$
  \triangle  = \{(u,v): u \ge 0, v \ge 0, u+v \le 1\}.
$$
For $\a,\b,\g > -1$, the classical Jacobi weight on the triangle is defined by 
\begin{equation}\label{eq:Jacobi-w}
     \sW_{\a,\b,\g}(u,v) = u^\a v^\b (1-u-v)^\g, \qquad \a,\b,\g > -1.
\end{equation}
Let $\la \cdot,\cdot\ra_{\a,\b,\g}$ be the inner product defined by 
$$
  \la f,g \ra_{\a,\b,\g} =\sb^\triangle_{\a,\b,\g} 
   \int_{\triangle} f(u,v) g(u,v)   \sW_{\a,\b,\g}(u,v) \d y\d v. 
$$
where $\sb^\triangle_{\a,\b,\g}$ is the normalization constant of $\sW_{\a,\b,\g}$,
\begin{equation}\label{eq:Jacobi-const}
\sb^\triangle_{\a,\b,\g} = \frac{\Gamma(\a+\b+\g+3)}{\Gamma(\a+1)\Gamma(\b+1)\Gamma(\g+1)}.
\end{equation}
Several orthogonal bases for $\CV_m(\sW_{\a,\b,\g}; \triangle)$ can be given explicitly in terms of the Jacobi polynomials. 
The triangle and the weight function $\sW_{\a,\b,\g}$ are symmetric under the simultaneous permutation of $(u,v,1-u-v)$ 
and $(\a,\b,\g)$. Hence, permuting $(x,y,1-x-y)$ and $(\a,\b,\g)$ simultaneously of an orthogonal basis leads to another 
orthogonal basis. We give three orthogonal bases and start with 
\begin{equation}\label{eq:triOP-T}
  \sT_{j,m}^{\a,\b,\g} (u,v) = P_{m-j}^{(2j+ \a+\g+1,\b)}(2 v -1) (1-v)^j P_j^{(\a,\g)}\left(1- \frac{2u}{1-v}\right), \quad 0 \le j \le m. 
\end{equation}
The set $\{\sT_{j,m}^{\a,\b,\g}: 0 \le j \le m\}$ is an orthogonal basis for $\CV_m(\sW_{\a,\b,\g}; \triangle)$. Permutation of 
variables and parameters simultaneously leads to two more bases. The first one is given by 
$\sS_{j,m}^{\a,\b,\g}(u,v)= \sT_{j,m}^{\g,\a,\b} (1- u-v,u)$ or, more explicitly, 
\begin{align}\label{eq:triOP-S}
  \sS_{j,m}^{\a,\b,\g} (u,v) = P_{m-j}^{(2j+ \b+\g+1,\a)}(2 u -1) (1-u)^j P_j^{(\g,\b)}\left( \frac{2v}{1-u} -1\right), \,\, 0 \le j \le m. 
\end{align}
The second one is given by $\sR_{j,m}^{\a,\b,\g} (u,v) =  \sT_{j,m}^{\b,\g,\a} (v, 1- u-v)$ or, more explicitly,
\begin{equation}\label{eq:triOP-R}
  \sR_{j,m}^{\a,\b,\g} (u,v) = P_{m-j}^{(2j+ \a+\b+1,\g)}(1-2 u -2v) (u+v)^j P_j^{(\b,\a)}\left( \frac{u-v}{u+v}\right), \quad 0 \le j \le m.
\end{equation} 

\begin{prop}
Each of the three sets $\{\sT_{j,m}^{\a,\b,\g}: 0 \le j \le m\}$, $\{\sS_{j,m}^{\a,\b,\g}: 0 \le j \le m\}$, and 
$\{\sR_{j,m}^{\a,\b,\g}: 0 \le j \le m\}$ is an orthogonal basis for $\CV_m(\sW_{\a,\b,\g}; \triangle)$. Moreover, 
\begin{align} \label{eq:hjm-triangle}
 \la\sT_{j,m}^{\a,\b,\g},\sT_{j,m}^{\a,\b,\g}\ra_{\a,\b,\g} 
   = &\, \frac{(\a+1)_{n-k} (\b+1)_k (\g+1)_k (\b+\g+2)_{n+k}} {(n-k)!k!(\b+\g+2)_k (\a+\b+\g+3)_{n+k}}\\
      &\times\frac{(n+k+\a+\b+\g+2)(k+\b+\g+1)}{(2n+\a+\b+\g+2)(2k+\b+\g+1)}
      =: \sh_{j,m}^{\a,\b,\g}  \notag
\end{align}
for $0 \le j \le m$ and 
\begin{align} \label{eq:hjm-triangle2}
\la\sS_{j,m}^{\a,\b,\g},\sS_{j,m}^{\a,\b,\g}\ra_{\a,\b,\g} =   \sh_{j,m}^{\g, \a,\b} \quad
\hbox{and} \quad \la\sR_{j,m}^{\a,\b,\g},\sR_{j,m}^{\a,\b,\g}\ra_{\a,\b,\g} = \sh_{j,m}^{\b,\g,\a}.
\end{align}
\end{prop}
 
The above bases are derived by the separation of variables. More generally, if the weight function $w$ satisfies 
$\sW(u,v) = w_1(v) w_2(\frac{u}{1-v})$, then we can derive an orthogonal basis for $\CV_n(\sW, \triangle)$ in 
terms of OPs with respect to $w_1$ and $w_2$ \cite{A, K}. We can consider, for example, 
\begin{equation} \label{eq:Jacobi-w+}
     \sW_{\a,\b,\g,\t}(u,v) = u^\a v^\b (1-u-v)^\g (1-v)^\t, \qquad \a,\b,\g,\t > -1. 
\end{equation}
This can be written as $w_1(v) w_2(\frac{u}{1-v})$, where $w_1(v) = v^\b (1-v)^{\a+\g+\t}$ and $w_2(v)= v^\a (1-v)^{\g}$.
With this separation of variables, we can define \cite{OTV}
\begin{equation}\label{eq:triOP-T2}
  \sT_{j,m}^{\a,\b,\g,\t} (u,v) = P_{m-j}^{(2j+ \a+\g+\t+1,\b)}(2 v -1) (1-v)^j P_j^{(\a,\g)}\left(1- \frac{2u}{1-v}\right)
\end{equation}
and conclude that $\{\sT_{j,m}^{\a,\b,\g,\t}: 0 \le j \le m\}$ is an orthogonal basis for $\CV_n(\sW_{\a,\b,\g,\t}, \triangle)$.

\subsection{Spherical harmonics and OPs on the unit ball}
These polynomials are building blocks of OPs on domains of revolution. We recall their definition and basic
properties; see \cite{DaiX, DX} for further results. 

\subsubsection{Spherical harmonics}
These are the restrictions of homogenous harmonic polynomials on the unit sphere $\sph$. Let 
$\CH_n^d$ be the space of spherical harmonics of degree $n$ of $d$ variables. It is known that
\begin{equation} \label{eq:dimHnd}
 \dim \CH_n^d = \binom{n+d-1}{n} - \binom{n+d-3}{n-2}.
\end{equation}
An orthogonal basis of $\CH_n^d$ can be given explicitly in terms of the Jacobi polynomials. The spherical 
harmonics are orthogonal with respect to the surface measure $\d \s$ on $\sph$. Let $\{Y_\ell^n\}$ be an
orthonormal basis of $\CH_n^d$. Then 
$$
   \frac{1}{\o_d} \int_{\sph} Y_\ell^n(\xi) Y_{\ell'}^{n'} (\xi) \d \s(\xi) =  \delta_{\ell,\ell'} \delta_{n,n'}, 
$$
where $\o_d$ denotes the surface area of $\sph$. In terms of this basis, the kernel $P_n(\cdot,\cdot)$ defined by 
$$
  P_n(x,y) = \sum_{1 \le \ell \le \dim \CH_n^d} Y_\ell^n(\xi) Y_\ell^n(\eta), \qquad \xi,\eta \in \sph, 
$$
is the reproducing kernel of $\CH_n^d$ in $L^2(\sph)$, which is the kernel of the orthogonal projection operator
$\proj: L^2(\sph) \mapsto \CH_n^d$. The kernel satisfies an addition formula 
$$
  P_n(\xi,\eta) = Z_n^{\f{d-2}{2}} (\la \xi,\eta\ra ), \qquad Z_n^\l (t) = \frac{n+\l}{\l}C_n^\l(t),
$$
where $C_n^\l$ is the Gegenbauer polynomial. Another important property is that spherical harmonics are 
eigenvalues of the Laplace-Beltrami operator $\Delta_0$, which is the restriction of the Laplace operator on 
the unit sphere; more precisely, 
$$
  \Delta_0 Y = - n(n+d-2) Y, \qquad \forall Y \in \CH_n^d, \quad n = 0,1,2,\ldots. 
$$
These are the two properties mentioned in the introduction and they play essential roles in the approximation
theory and harmonic analysis on the unit sphere. 

\subsubsection{Classical OPs on the unit ball $\BB^d$} These are OPs orthogonal with respect the weigh function 
\begin{equation}\label{eq:weightB}
  W_\mu(x) = (1-\|x\|)^{\mu-\f12}, \qquad x\in \BB^d, \quad \mu > -\tfrac12
\end{equation}
on the unit ball $\BB^d$. The normalization constant of $W_\mu$ is 
$b_\mu^\BB = \Gamma(\mu+\f{d+1}{2}) /(\pi^{\f d 2}\Gamma(\mu+\f12))$. Let $\CV_n(W_\mu, \BB^d)$ be the
space of OPs of degree $n$. An orthogonal basis of this space can be given explicitly in terms of the Jacobi 
polynomials and spherical harmonics. For $ 0 \le m \le n/2$, let $\{Y_\ell^{n-2m}: 1 \le \ell \le  \dim \CH_{n-2m}\}$ 
be an orthonormal basis of $\CH_{n-2m}^d$. Define \cite[(5.2.4)]{DX}
\begin{equation}\label{eq:basisBd}
  P_{\ell, m}^n (W_\mu; x) = P_m^{(\mu-\f12, n-2m+\f{d-2}{2})} \left(2\|x\|^2-1\right) Y_{\ell}^{n-2m}(x).
\end{equation}
Then $\{P_{\ell,m}^n(W_\mu): 0 \le m \le n/2, 1 \le \ell \le \dim \CH_{n-2m}\}$ is an orthogonal basis of $\CV_n^d(W_\mu,\BB^d)$. 
Let $\Hb_{m,n}^\mu$ be the square of the norm of $P_{\ell,m}^n(W_\mu)$ in $L^2(\Wb_\mu; \BB^d)$. Then the reproducing 
kernel of the space $\CV_n(W_\mu,\BB^d)$ can be written as
$$
 P_n^\mu(x,y) = \sum_{0 \le m \le n/2} \sum_{ 1 \le \ell \le \dim \CH_{n-2m}} 
      \frac{P_{\ell, m}^n(W_\mu; x) P_{\ell, m}^n(W_\mu;y)}{\Hb_{m,n}^\mu}.
$$
This is also the kernel of the projection operator $\proj_n^\mu: L^2(W_\mu; \BB^d) \mapsto \CV_n(W_\mu,\BB^d)$
and it satisfies an addition formula \cite{X99}:
$$
  P_n^\mu(x,y) = c_\mu \int_{-1}^1 Z_n^{\mu+\f{d-1}{2}} \left (\la x,y\ra + t \sqrt{1-\|x\|^2}\sqrt{1-\|y\|^2}\right) (1-t^2)^{\mu-1} \d t
$$ 
for $\mu > 0$ and in limit for $\mu = 0$. Moreover, there is a spectral operator that has OPs as eigenfunctions
(\cite[(5.23)]{DX}), 
$$
 \left (\Delta - \la x ,\nabla \ra^2 - (2\mu + d-1) \la x, \nabla\ra \right ) u = - n(n+2\mu+d-1), \quad \forall u \in  \CV_n(\BB^d, W_\mu).
$$
As in the case of spherical harmonics, these two properties provide powerful tools for the approximation and harmonic analysis 
on the unit ball (cf. \cite{DaiX, DX}). 

\subsection{OPs on quadratic domains of revolution} 
Let $\phi$ be either a nonnegative linear polynomial or the square root of a nonnegative quadratic polynomial 
on the interval $[a,b]$, which can be an infinite interval, say $\RR_+$ or $\RR$. We consider the solid domain 
of revolution defined by 
$$
\mathbb{V}^{d+1} = \{(x,t): \|x\| \le \phi(t), \, x \in \mathbb{R}^d, \, t \in [a,b]\}.
$$
For $\phi(t)=1$, the domain is a cylinder. For $\phi(t) = \sqrt{1-t^2}$ on $[-1,1]$, the domain becomes the unit ball. We 
consider OPs for the family of weight functions of the form $\Wb(x,t) = w(t) (t^2-\|x\|^2)^{\mu-\f12}$. For several domains, 
explicit orthogonal bases for some $w(t)$ can be given in terms of known polynomials. We recall two particular cases 
that are most relevant to our study in this paper. 

\subsubsection{Jacobi polynomials on the cone} \label{subsect:cone}
Here $\phi(t) = t$ and the weight function is
$$
  \Wb_{\g,\mu}(x,t) = (1-t)^\g (t^2-\|x\|^2)^{\mu-\f12}, \qquad  0 \le t \le 1, \quad \g > -1, \,\,\mu > -\tfrac12.
$$
An orthogonal basis of $\CV_n(W_{\b,\g,\mu},\VV^{d+1})$ can be given 
in terms of the Jacobi polynomials and the OPs on the unit ball. For $m =0,1,2,\ldots$, 
let $\{P_\kb^m(W_\mu): |\kb| = m, \kb \in \NN_0^d\}$ be an orthonormal basis of $\CV_n(W_\mu,\BB^d)$ 
on the unit ball. Then the polynomials 
\begin{equation} \label{eq:coneJ}
  \Jb_{m,\kb}^n(x,t):= P_{n-m}^{(2\mu+2m+d-1, \g)}(1- 2t) t^m P_\kb^m\left(W_\mu; \frac{x}{t}\right), \quad
   |\kb| = m, \,\, 0 \le m \le n, 
\end{equation}
consist of an orthogonal basis of $\CV_n(\Wb_{\g,\mu},\VV^{d+1})$, which were called the Jacobi polynomials
in \cite{X20}. The reproducing kernel $\Pb_n \big(\Wb_{\g,\mu})$ of the space $\CV_n(\Wb_{\g,\mu},\VV^{d+1})$
satisfies an addition formula for $\mu \ge 0$ and $\g \ge -\f12$, 
\begin{align}\label{eq:PbCone}
  \Pb_n \big(\Wb_{\g,\mu}; (x,t), (y,s)\big) =\, & 
   c_{\mu,\g,d} \int_{[-1,1]^3}  Z_{2n}^{2 \mu+\g+d} (\xi (x, t, y, s; u, v)) \\
     &\times   (1-u^2)^{\mu-1} (1-v_1^2)^{\mu+\frac{d-3}2}(1-v_2^2)^{\g-\f12}  \d u \d v, \notag
\end{align}
where $c_{\mu,\g,d}$ is a constant, so that $\Pb_0 =1$ and 
$\xi (x,t, y,s; u, v) \in [-1,1]$ is defined by 
\begin{align} \label{eq:xi}
\xi (x,t, y,s; u, v) = &\, v_1 \sqrt{\tfrac12 \left(ts+\la x,y \ra + \sqrt{t^2-\|x\|^2} \sqrt{s^2-\|y\|^2} \, u \right)}\\
      & + v_2 \sqrt{1-t}\sqrt{1-s}, \notag
\end{align}
and the formula holds under limit when $\mu = 0$ and/or $\g = -\f12$. Moreover, for $\mu > -\tfrac12$, $\g > -1$
the second-order differential operator 
\begin{align*}
  \mathfrak{D}_{\g,\mu} : = & \, t(1-t)\partial_t^2 + 2 (1-t) \la x,\nabla_x \ra \partial_t + \sum_{i=1}^d(t - x_i^2) \partial_{x_i}^2
        - 2 \sum_{i<j } x_i x_j \partial_{x_i} \partial_{x_j}  \\ 
  &   + (2\mu+d)\partial_t  - (2\mu+\g+d+1)( \la x,\nabla_x\ra + t \partial_t),  
\end{align*}
where $\nabla_x$ and $\Delta_x$ denote the gradient and the Laplace operator in $x$-variable, has the
polynomials in $\CV_n(\Wb_{\g,\mu},VV^{d+1})$ are eigenfunctions; more precisely, 
\begin{equation}\label{eq:cone-eigen}
   \mathfrak{D}_{\g,\mu} u =  -n (n+2\mu+\g+d) u, \qquad \forall u \in \CV_n(\Wb_{\g,\mu},\VV^{d+1}).
\end{equation}
A classification of such spectral operators is known for $d=2$ (\cite{KS}) but not for $d > 2$. The operator 
$\mathfrak{D}_{\g,\mu}$ on the cone is a recent addition for $d > 2$. 

\subsubsection{Gegenbauer polynomials on the hyperboloid and double cone.} \label{subset:2.4.2}
For $\varrho > 0$, $\phi(t) =  \sqrt{t^2 + \rho^2}$, $0 < |\rho| \le |t|$. The domain is defined by 
$$
  \XX^{d+1} =  \left \{(x,t): \|x\|^2 \le t^2 - \varrho^2, \, x \in \RR^d, \, \varrho \le |t| \le \sqrt{\varrho^2 +1}\right\}, 
$$
which is a hyperboloid for $\varrho > 0$ and degenerates to a double cone for $\varrho = 0$. For $\b >  -\f12$, 
$\g > - \f12$ and $\mu > -\f12$, we consider the weight function defined by
$$
  \Wb_{\g,\mu}(x,t): = |t|(t^2-\varrho^2)^{-\f12} (1+\varrho^2-t^2)^{\g-\f12}(t^2-\varrho^2-\|x\|^2)^{\mu - \f12}.
$$
The OPs associated with $\Wb_{\g,\mu}$ are called the Gegenbauer polynomials on the hyperboloid in \cite{X20}.
Since the weight function is even in the $t$ variable, the space $\CV_n( \Wb_{\g,\mu},\XX^{d+1})$ naturally split into
two parts depending on the parity of OPs. The space of OPs that consists of the Gegenbauer polynomials even in the 
$t$ variable possesses both addition formula and spectral properties \cite{X20}. These will be needed later in the paper 
and will be reviewed in the Subsection \ref{sec:Gegen-cone}.

\section{Orthogonal polynomials on domains of revolution in $\RR^{d+1}$}
\setcounter{equation}{0}

In the first subsection, we discuss our first but most essential construction of OPs for domains of 
revolution. Several examples are given in the second subsection. Further construction and examples will be 
given in later sections.

\subsection{Orthogonal structure on domains of revolution}

Let $\Omega$ be a domain of $\RR^2$, which we decompose as $\Omega = \Omega_+ \cup \Omega_-$, where 
$$
\Omega_+:= \{(s,t) \in \Omega: s \ge 0\} \quad \hbox{and} \quad \Omega_-:= \{(s,t) \in \Omega: s \le 0\}.
$$
We assume that $\Omega$ is symmetric in the $s$-variable; that is, $(s,t) \in \Omega_+$ if and only if $(-s,t) \in \Omega_-$. 
Our goal is to consider the domain of revolution defined by
$$
  \VV^{d+1} = \left \{(x,t) \in \RR^{d+1}, \quad x \in \RR^d,\, t \in \RR, \quad (\|x\|,t) \in \Omega_+\right\}.
$$
In words, $\VV^{d+1}$ is the domain of revolving $\Omega_+$ of $d$ variables around the $t$ axis. 

Let $\sW(s,t)$ be a weight function defined on $\Omega_+$ and we assume that it is nonnegative on $\Omega_+$ and 
has finite moments. Let\footnote{From now on, we adopt the convention that the letters 
$\Pb, \Qb, \Sb, \Wb$ etc., in the bold font, are reserved for functions and polynomials on domains of revolution.}   
$$
  \Wb(x,t):= \sW \left(\|x\|,t\right), \qquad (x,t) \in \VV^{d+1}.
$$
We consider OPs on $\VV^{d+1}$ with respect to the inner product 
$$
      \la f, g\ra_{\Wb} = \int_{\VV^{d+1}} f(x,t) g(x,t) \Wb(x,t) \d x \d t. 
$$
The assumption on $\sW$ implies the existence of OPs. Let $\Pi_n^{d+1}$ denote the space of polynomials
of total degree $n$ in $d+1$ variables. We denote by $\CV_n(\Wb; \VV^{d+1})$ the subspace of OPs of total
degree $n$ with respect to the inner product $\la \cdot,\cdot\ra_\Wb$. It is known that 
$$
  \dim \CV_n(\Wb, \VV^{d+1}) = \binom{n+d}{n}\quad \hbox{and} \quad \dim \Pi_n^{d+1} = \binom{n+d+1}{n}. 
$$

Our first goal is to show that an orthogonal basis for $\CV_n(\Wb,\VV^{d+1})$ can be derived from OPs of two 
variables on $\Omega_+$ and spherical harmonics. This is based on the observation that the integral over
$\VV^{d+1}$ can be decomposed as a double integral over $\Omega_+$ and the sphere $\sph$ of $\RR^d$. Indeed, 
setting $x = s \xi$ with $\xi \in \sph$ and $t \in \RR$, we can write $(x,t) = (s\xi,t)$, so that $(s,t) \in \Omega$. Then
\begin{align}\label{eq:intV=}
  \int_{\VV^{d+1}} f(x,t) \d x \d t \,& =   \int_{\Omega_+} \int_{\|x\| =s} f(x,t) \d x \, \d t\\
                  & =\int_{\sph} \int_{\Omega_+}  f(s \xi,t)  s^{d-1} \d s\, \d t \,\d \s(\xi), \notag
\end{align}
where $\d \s$ denote the surface measure on $\sph$. 

Let $\CV_n (\sW,\Omega)$ be the space of OPs of two variables with respect to the inner product
$$
      \la f,g\ra_\sW = \int_\Omega f(s,t) g(s,t) \sW(s,t) \d s \d t.
$$
We further denote by $\CV_n^\sE(\sW, \Omega)$ the subspace of polynomials that are even in the $s$ variable; in other
words, 
$$
\CV_n^\sE(\sW, \Omega) =\left \{ P \in \CV_n(\sW, \Omega): P(s,t) = P(-s,t) \right\}.
$$
The elements of these spaces are polynomials in two variables and it is easily seen that 
$$
  \dim \CV_n(\sW,\Omega) = n+1 \quad \hbox{and} \quad \dim \CV_n^\sE(\sW,\Omega) = \left \lfloor \frac{n}{2} \right \rfloor + 1.
$$
Since $\sW(s,t)$ is even in $s$ variable, it follows readily that 
$$
    \la f,g\ra_\sW = \frac12 \int_{\Omega_+} f(s,t) g(s,t) \sW(s,t) \d s \d t
$$
for all polynomials $f$ and $g$ that are even in their first variable. Let $k$ be a positive integer. We define 
\begin{equation} \label{eq:sWk}
     \sW^{(k)}(s,t) = |s|^{k+d-1} \sW(s,t), \qquad (s,t) \in \Omega. 
\end{equation}

\begin{thm} \label{thm:OP_V}
For $k \ge 0$, let $\{\sP_j^{m}\left(\sW^{(2k)}; s, t\right): 0\le j \le \left \lfloor \f m 2 \right\rfloor\}$ denote an 
orthogonal basis of $\CV_m^\sE\left(\sW^{(2k)}, \Omega\right)$. Let $\{Y_{\ell}^k: 1 \le \ell \le \dim \CH_k^d\}$ be an 
orthonormal basis of the space $\CH_k^d$ of spherical harmonics. Define 
$$
     \Qb_{j,k,\ell}^n (x,t) = \sP_j^{n-k} \left(\sW^{(2k)}; \|x\|, t\right) Y_\ell^k (x). 
$$
Then the set $\left \{\Qb_{j,k,\ell}^n:  1 \le \ell \le \dim \CH_k^d, \,0\le j \le \left \lfloor \frac{n-k}{2} \right \rfloor, \, 0 \le k \le n\right\}$
is an orthogonal basis of $\CV_n(\Wb,\VV^{d+1})$ for $\Wb(x) = \sW(\|x\|,t)$. Moreover,
\begin{equation}\label{eq:OP_V_norm}
 \Hb_{j,k}^n := \left \langle \Qb_{j,k,\ell}^n,  \Qb_{j',k',\ell'}^{n'} \right \rangle_\Wb 
     = \left \langle \sP_j^{n-k}\left(\sW^{(2k)}\right),\sP_j^{n-k}\left(\sW^{(2k)}\right) \right \rangle_{\sW^{(2k)}}=: \sH_{j,k}^{n-k}. 
\end{equation}
\end{thm}

\begin{proof}
Since $\sP_j^{n-k}\left(\sW^{(2k)}\right)$ is even in its first variable, we can write 
$$
  \sP_j^{n-k}\left(\sW^{(2k)}; s,t\right) = \frac12 \left[\sP_j^{n-k}\left(\sW^{(2k)}; s, t\right)+\sP_j^{n-k}\left(\sW^{(2k)}; -s, t\right) \right],
$$
which implies immediately that it contains only even powers of $s$ and, consequently, $\sP_j^{n-k}\left(\sW^{(2k)}; \|x\|,t\right)$ 
is a polynomial of degree $n-k$ in $(x,t)$ variables. Thus, $\Qb_{j, k, \ell}^n$ is a polynomial of degree $n$ in $(x,t)$ variables. 
Since $Y_\ell^k$ is homogeneous, we can write 
$$
   \Qb_{j,k,\ell}^n(x,t) = \sP_j^{n-k}\left(\sW^{(2k)}; s, t\right) s^k Y_\ell^k(\xi), \quad   x = s \xi, \quad \xi \in \sph.
$$
Using the orthogonality of spherical harmonics and the identity \eqref{eq:intV=}, we obtain
\begin{align*}
\left \langle \Qb_{j,k,\ell}^n,  \Qb_{j',k',\ell'}^{n'} \right \rangle_\Wb& = \delta_{\ell,\ell'} \delta_{k,k'} 
     \int_{\Omega}  \sP_j^{n-k}\left(\sW^{(2k)};s, t\right) \sP_{j'}^{n'-k}\left(\sW^{(2k)};s, t\right)
          \sW^{(2k)}(s,t) \d s \d t \\
       & = h_{j,k}^{n-k} \delta_{\ell,\ell'} \delta_{k,k'} \delta_{j,j'} \delta_{n,n'}, 
\end{align*}
where $h_{j,k}^{n-k}$ is the norm square of $\sP_j^{n-k}\left(\sW^{(2k)}\right)$. This proves the orthogonality and that 
$\Qb_{j, k, \ell}^n \in \CV_n(\Wb, \VV^{d+1})$. To complete the proof, we need to show that the cardinality of 
$\{\Qb_{j, k, \ell}^n\}$ is equal to $\dim \CV_n(\Wb, \VV^{d+1})$; that is,  
$$
  \sum_{k=0}^n \left(  \left \lfloor \frac{n-k}{2} \right \rfloor+1\right) \dim \CH_k^d = \binom{n+d}{n}.
$$
By \eqref{eq:dimHnd}, the left-hand side of the above equation is equal to 
\begin{align*}
&  \sum_{k=0}^n \left(  \left \lfloor \frac{n-k}{2} \right \rfloor+1\right) \binom{k+d-1}{k} -
   \sum_{k=0}^{n-2} \left(  \left \lfloor \frac{n-k}{2}\right \rfloor \right) \binom{k+d-1}{k} \\
  & \qquad =  \sum_{k=0}^{n-2} \binom{k+d-1}{k} + \binom{n+d-2}{n-1} + \binom{n+d-1}{n} = \binom{n+d}{n},
\end{align*}
where the last identity follows from a straightforward computation. This completes the proof.
\end{proof}

After a brief subsection on the reproducing kernels, we will give several examples for which an explicit basis 
for $\CV_n^\sE(\sW, \Omega)$ can be derived and, as a consequence, so can an explicit basis for $\CV_n(\Wb, \VV^{d+1})$.

\subsection{Reproducing kernels and orthogonal series} \label{subsect:kernel}
The Fourier orthogonal expansion of $f \in L^2(\Wb,\VV^{d+1})$ can be written as 
$$
  f = \sum_{n=0}^\infty \proj_n (\Wb; f), \qquad \proj(\Wb): L^2(\Wb,\VV^{d+1}) \mapsto \CV_n(\Wb,\VV^{d+1}),
$$  
where $\proj_n(\Wb)$ is the orthogonal projection operator. In terms of the orthogonal basis $\{\Qb_{j,k,\ell}^n\}$ given in 
Theorem \ref{thm:OP_V}, we can write 
$$  
\proj_n(\Wb; f):= \sum_{k=0}^n \sum_{j=0}^{\left \lfloor \frac{n-k}{2} \right \rfloor} \sum_{\ell=1}^{\dim \CH_k^d} 
\wh f_{j,k,\ell}^n \Qb_{j,k,\ell}^n, \qquad \wh f_{j,k,\ell}^n = \frac{\langle f, \Qb_{j,k,\ell}^n \rangle_\Wb}{\Hb_{j,k}^n}.
$$
Let $\Pb_n(\Wb; \cdot,\cdot)$ be the reproducing kernel of the space $\CV_n(\Wb,\VV^{d+1})$, which is uniquely 
determined by 
$$
    \int_{\VV^{d+1}} \Pb_n(\Wb; (x,t),(y,s)) \Qb (y,s) \Wb(y,s) \d y \d s = \Qb (x,t), \qquad \forall \Qb \in \CV_n(\Wb,\VV^{d+1}).
$$ 
The reproducing kernel satisfies, in terms of the orthogonal basis $\{\Qb_{j,k,\ell}^n\}$, 
\begin{align}
  &  \Pb_n(\Wb; (x,t),(y,s)) = \sum_{k=0}^n \sum_{j=0}^{\left \lfloor \frac{n-k}{2} \right \rfloor} \sum_{\ell=1}^{\dim \CH_k^d}
  \frac{\Qb_{j,k,\ell}^n(x,t) \Qb_{j,k,\ell}^n(y,s)}{\Hb_{j,k}^n} \\
  &  \quad = \sum_{k=0}^n \sum_{j=0}^{\left \lfloor \frac{n-k}{2} \right \rfloor} \sum_{\ell=1}^{\dim \CH_k^d}
  \frac{\sP_j^{n-k}(\Wb; \|x\|,t)\sP_j^{n-k}(\Wb; \|y\|,s)}{\sH_{j,k}^{n-k}} \|x\|^k \|y\|^k Z_k^{\frac{d-2}{2}} (\la\xi,\eta \ra), \notag
\end{align}
where the second identity follows from the addition formula for spherical harmonics. An addition formula holds for the kernel
if the right-hand sums can be written in a closed form. Moreover, the projection operator is an
integral operator with $\Pb_n(\Wb; \cdot,\cdot)$ as its kernel, 
$$
  \proj_n (\Wb;f) = \int_{\VV^{d+1}} \Pb_n(\Wb; (x,t),(y,s)) f(y,s) \Wb(y,s) \d y \d s. 
$$
This relation is the reason why an addition formula is a powerful tool for studying the Fourier orthogonal expansions. 

\subsection{Cylinder} 
Our first example is the simplest one, for which $\Omega$ is a rectangular domain. We consider as an example the 
square with weight function 
$$
\Omega  = [-1,1]^2 \quad \hbox{and} \quad  \sW(s,t) = s^{2\a} (1-s^2)^{\mu-\f12}(1-t^2)^{\l-\f12},
$$
so that its rotation in the $t$ axis leads to the cylinder and a weight function given by
\begin{equation} \label{eq:cylinder}
   \VV^{d+1} = \{(x,t): (\|x\|,t) \in \Omega\}  \quad \hbox{and} \quad  \Wb(x,t) =  \|x\|^{2\a} (1-\|x\|^2)^{\mu-\f12}(1-t^2)^{\l-\f12},
\end{equation}
where $\a > -1$, $\l, \mu > -\f12$. 

One orthogonal basis for $\CV_n(\sW, \Omega)$ consists of products of Gegenbauer polynomials and 
generalized Gegenbauer polynomials
$$
  \sP_{k,n}^{\a,\mu,\l} (s,t) = C_{n-k}^{\l}(t) C_k^{(\mu,\a)}(s), \quad 0 \le k \le n. 
$$
The polynomial $C_k^{(\mu,\a)}$ has the same parity as $k$ and it satisfies, in particular, 
\begin{equation} \label{eq:GGegen-J}
C_{2j}^{(\mu,\a)}(s) = \mathrm{const.}  P_{j}^{(\mu-\f12, \a -\f12)}(2s^2-1).
\end{equation}
Taking into consideration of degrees, we see that $\{\sP_{n-2j,2j}^{\a,\mu,\l}: 0 \le j \le n/2\}$ is an orthogonal basis of 
$\CV_n^\sE(\sW,\Omega)$. In particular, let $Y_\ell^k$ denote an orthogonal basis of $\CH_k^d$ and recall 
$\sW^{(2k)}(t) = s^{2k+d-1}\sW(s,t)$; then the construction in Theorem \ref{thm:OP_V} shows that 
$$
   \sP_{n-k-2j,2j}^{\a+k+\f{d-1}{2},\mu,\l}(\|x\|,t) Y_{\ell}^{k}(x), \quad  0 \le j \le \left\lfloor \tfrac{n-k}2\right\rfloor, 
  \quad 0 \le k \le n,
$$
is an orthogonal basis of $\CV_n(\Wb, \VV^{d+1})$. Changing index $k = m - 2j$ shows that the basis consists of 
\begin{equation} \label{eq:cylinder2}
  \Qb_{m,j,\ell}^{n; (\a,\mu,\l)}(x,t) = C_{n-m}^\l (t) P_{j}^{(\mu-\f12, \a+m-2j+\f{d-2}2)}(2\|x\|^2-1)Y_{\ell}^{m-2j}(x). 
\end{equation}
We summarize this in the following proposition.

\begin{prop}
Let $\{Y_{\ell}^{k}: 1 \le \ell \le \dim \CH_k^d\}$ be an orthogonal basis of $\CH_k^d$. Then 
$$
\left\{\Qb_{m,j,\ell}^{n; (\a,\mu,\l)}(x,t): 1 \le \ell \le \dim \CH_{m-2j}^d, \, 0 \le j \le m/2, \, 0\le m\le n\right\}
$$ 
consists of an orthogonal basis for $\CV_n(\Wb, \VV^{d+1})$ on the cylinder domain \eqref{eq:cylinder}.
\end{prop}

If $\a = 0$, then the product of the last two terms in \eqref{eq:cylinder2} is exactly the polynomial $P_{\ell,j}^m(W_\mu;x)$
for the classical OPs on the ball \eqref{eq:basisBd}, so that the basis takes the form
$$
 \Qb_{m,j,\ell}^{n; (0,\mu,\l)}(x,t) = C_{n-m}^\l(t) P_{\ell,j}^m (W_\mu;x), \quad 1 \le \ell \le \dim \CH_k^d, 
   \,0\le j \le \left \lfloor \tfrac{m}{2} \right \rfloor, \, 0 \le m \le n,
$$
which is the usual orthogonal basis for $\CV_n(\Wb, \VV^{d+1})$ on the cylinder domain (cf. \cite{DX, W}). 

\subsection{Conic domains}
We consider the case when $\Omega$ is the triangle symmetric in the $s$ variable. There are two cases. 

\subsubsection{Cone}
In this example, $\Omega$ is the triangle $\Omega_\vartriangle$ defined by   
$$
\Omega_\vartriangle = \{(s,t): |s| \le t, 0 \le t \le 1\} \quad \hbox{and} \quad  \sW(s,t) = |s|^{2\a}(t^2 -s^2)^{\mu-\f12}t^\b(1-t)^\g. 
$$
The rotation of $\Omega_\vartriangle$ in the $t$ axis leads to the solid cone 
$$
   \VV_\vartriangle^{d+1} = \{(x,t): (\|x\|,t) \in\Omega_\vartriangle\} =  \{(x,t): \|x\|\le t, \, x \in \RR^d,\, 0 \le t \le 1\} 
$$
equipped with the weight function 
$$
\Wb_\vartriangle(x,t) = \sW(\|x\|,t) =  \|x\|^{2\a} (t^2-\|x\|^2)^{\mu-\f12}t^\b (1-t)^{\g}. 
$$
The triangle domain $\Omega$ and the cone are depicted in Figure \ref{fig:Cone}, where $\Omega_+$ is shaded. 

\begin{figure}[htb] \label{fig-cone}
\hfill\begin{minipage}[b]{0.4\textwidth} \centering
\includegraphics[width=1.1\textwidth]{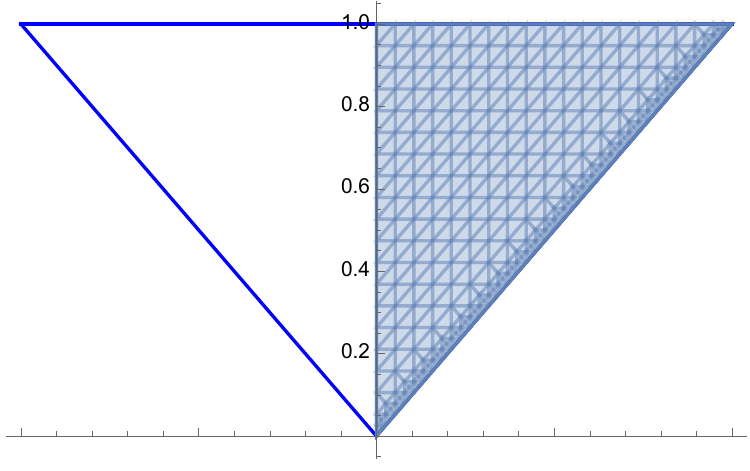}
\end{minipage}\hfill\begin{minipage}[b]{0.5\textwidth}
\centering
\includegraphics[width=1.05\textwidth]{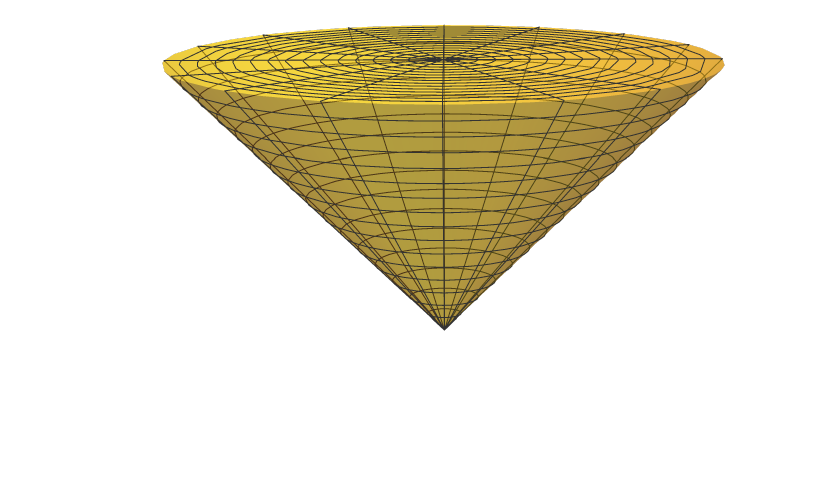}
\end{minipage}\hspace*{\fill}
\caption{Triangle and cone} \label{fig:Cone}
\end{figure}

For $\a = 0$, OPs on the cone have been studied in \cite{X20}. We shall carry out the construction in our new point of
view. Making a change of variable $s \to t u$, the integral over $\Omega_\vartriangle$ can be written as 
\begin{align*}
  \int_{\Omega_\vartriangle} f(s,t) \sW(s,t) \d s \d t\, & = \int_0^1 \int_{-1}^1 f(t u, t) \sW(t u, t) t \d u\,\d t    \\
     &  = \int_{0}^1 t^{2 \a + \b+ 2\mu} (1-t)^\g \int_{-1}^1 f(t u, t) |u|^\a (1-u^2)^{\mu -\f12} \d u \d t,
\end{align*}
from which it is easy to verify that an orthogonal basis for $\CV_m(\sW, \Omega_\vartriangle)$ is given by 
$$
    \sP_{k}^m(s,t) = P_{m-k}^{(2k+\a +\b+ 2\mu,\g)}(1-2t) t^k C_k^{(\mu, \a)}\left(\frac s t\right), \quad 0 \le k \le m,
$$
in terms of the Jacobi polynomials and the generalized Gegenbauer polynomials of degree $m$. Hence, using 
\eqref{eq:GGegen-J} as in the cylinder case, we conclude that an orthogonal basis for 
$\CV_{m}^{\sE}(\sW^{(2k)}, \Omega_\vartriangle)$ is given by 
$$
   \sP_j^{m}(s,t) = P_{m- 2j}^{(4j +2k+2\a +\b+ 2\mu+d-1,\g)}(1-2t) t^{2j} P_{j}^{(\mu-\f12, k+\a +\f{d-2}{2})}\left(2 \frac{s^2}{t^2}-1\right) 
$$
with $0 \le j \le \left\lfloor \f{m}2\right\rfloor$. In particular, the construction in Theorem \ref{thm:OP_V} 
shows that an orthogonal basis for $\CV_n(\Wb_\vartriangle, \VV_\vartriangle^{d+1})$ is given by 
\begin{align*}
 P_{n-k- 2j}^{(4j +2k+\a +\b+ 2\mu+d-1,\g)}(1-2t) t^{2j} P_{j}^{(\mu-\f12, k+\f{\a+d-2}{2})}\left(2 \frac{\|x\|^2}{t^2}-1\right)Y_\ell^k (x), 
\end{align*}
where $1 \le \ell \le \dim \CH_k^d, \,0\le j \le \left \lfloor \frac{n-k}{2} \right \rfloor, \, 0 \le k \le n$, and changing index $k = m - 2j$ 
shows that the basis consists of 
$$
 \Qb_{j,m,\ell}^n (x,t) = P_{n-m}^{(2m+\b+ 2\a+ 2\mu+d-1,\g)}(1-2t) P_{j}^{(\mu-\f12, \a+m-2j+\f{d-2}2)}(2\|x\|^2-1)Y_{\ell}^{m-2j}(x).   
$$
We summarize this in the following proposition.

\begin{prop}
Let $\{Y_{\ell}^{k}: 1 \le \ell \le \dim \CH_k^d\}$ be an orthogonal basis of $\CH_k^d$. Then 
$$
\left\{\Qb_{m,j,\ell}^{n}(x,t): 1 \le \ell \le \dim \CH_{m-2j}^d, \, 0 \le j \le m/2, \, 0\le m\le n\right\}
$$ 
consists of an orthogonal basis for  $\CV_n(\Wb_\vartriangle, \VV_\vartriangle^{d+1})$ on the cone.
\end{prop}

In particular, if $\a = 0$, then the polynomial $\Qb_{j,m,\ell}^n$ can be written in terms of OPs 
$P_{j,\ell}^m(W_\mu)$ in \eqref{eq:basisBd} on the unit ball $\BB^d$; more precisely,
$$
 \Qb_{j,m,\ell}^n (x,t) = P_{n-m}^{(2m+\b+ 2\mu+d-1,\g)}(1-2t) P_{j,\ell}^m\left(W_\mu; \frac x t \right).  
$$
This last case is the orthogonal basis for $\CV_n(\Wb_\vartriangle; \VV_\vartriangle^{d+1})$ first derived and studied in
\cite{X20}, which satisfies several distinguished properties, including the spectral operator  \eqref{eq:cone-eigen} 
and the addition formula \eqref{eq:PbCone} stated in Subsectioin \ref{subsect:cone}. 

\subsubsection{Coupled cone}
Here we assume that $\Omega$ is a diamond symmetric in the $s$ variable, which we denote as $\Diamond$. We consider the 
setting 
$$
     \Diamond = \{(s,t): |s|+|t| \le 1\} 
$$
with the weight function 
$$
\sW(s,t) = |s|^{2\a+1}|t|^{2\b+1}\left(1-(s+t)^2\right)^\g \left(1-(s-t)^2\right)^\g,
$$
Then the rotation in the $t$ axis of the diamond domain gives the coupled cone 
$$
  \VV_\diamond^{d+1} = \{(x,t): (\|x\|,t)\in \Omega_\diamond\} = \left \{(x,t): \|x\| \le 1-t \le 1, \, x \in \RR^d\right\}
$$
with the weight function defined by
$$
  \Wb_\diamond(x,t) = \|x\|^{2\a+1}t^{2\b+1}\left(1-(\|x\|+t)^2\right)^{-\f12} \left(1-(\|x\|-t)^2\right)^{-\f12}. 
$$
The diamond domain $\Omega$ and the coupled cone $\VV^3$ in 3D are depicted in Figure \ref{fig:CoupledCone}.

\begin{figure}[htb]
\hfill\begin{minipage}[b]{0.44\textwidth} \centering
\includegraphics[width=1\textwidth]{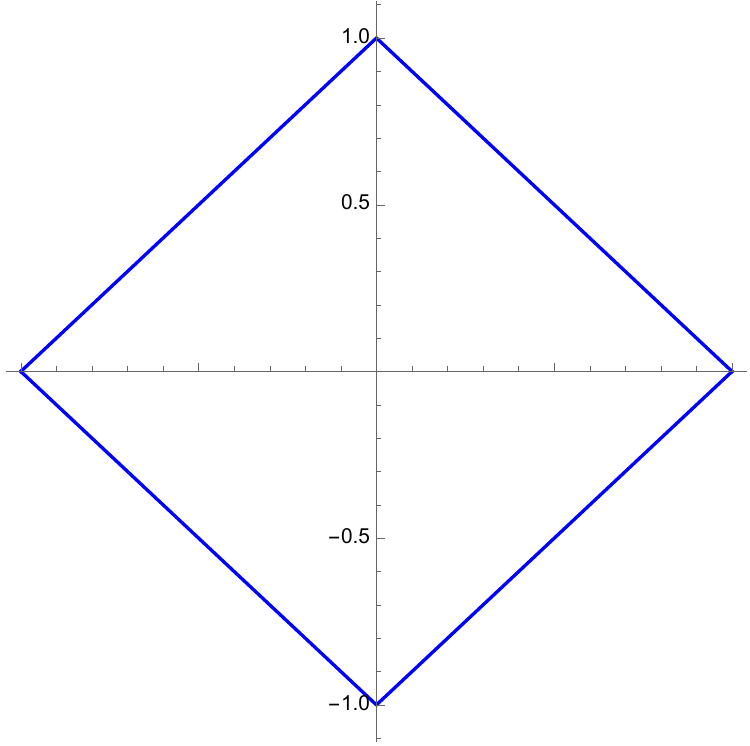}
\end{minipage}\hfill\begin{minipage}[b]{0.56\textwidth}
\centering
\includegraphics[width=0.75\textwidth]{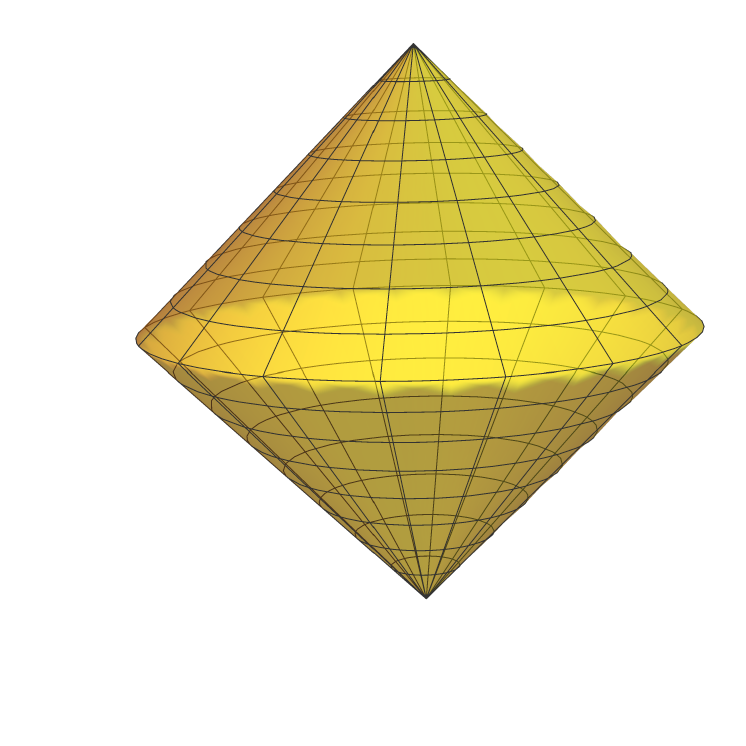}
\end{minipage}\hspace*{\fill}
\caption{Diamond and coupled cones}\label{fig:CoupledCone}
\end{figure}

Changing variables $u = t-s$ and $v= t+s$, the diamond domain becomes $[-1,1]^2$ in the $(u,v)$ plane and the weight 
function $\sW$ becomes 
$$
  \wh \sW(u,v) = \frac{1}{2^{2\a}} |u-v|^{2\a+1} |u+v|^{2\b+1}\left(1- u^2\right)^{-\f12} \left(1- v^2\right)^{-\f12}.
$$
An orthogonal basis for $\CV_n(\wh \sW, [-1,1]^2)$ is derived in \cite{X12} in terms of the Jacobi polynomials. For our 
purpose, we only need those that will lead to a basis for $\CV_m^\sE(\sW, \Diamond)$. These are given in 
\cite[Proposition 4.6]{X12} with $u = \cos \t$ and $v = \cos \phi$ as
\begin{align} 
 \sP_{j,2m}^{(\a,\b)} (u,v)& = P_m^{(\a,\b)}(\cos (\t-\phi)) P_j^{(\a,\b)}(\cos (\t+\phi))  \label{eq:squareOP1}\\
   & \qquad   + P_j^{(\a,\b)}(\cos (\t-\phi)) P_m^{(\a,\b)}(\cos (\t+\phi)),\quad 0 \le j \le m, \notag \\ 
 \sP_{j,2m+1}^{(\a,\b)} (u,v)& = (u+v) \left[P_m^{(\a,\b+1)}(\cos (\t-\phi)) P_j^{(\a,\b+1)}(\cos (\t+\phi)) \right. \label{eq:squareOP2}\\
     & \qquad   \left. + P_j^{(\a,\b+1)}(\cos (\t-\phi)) P_m^{(\a,\b+1)}(\cos (\t+\phi)) \right], \quad 0 \le j \le m, \notag
\end{align}
which are indeed polynomials of degree $n$ in $(u,v)$ variables, where $n = 2m$ or $2m+1$. Setting 
$t-s= u = \cos \t$ and $t+s = v = \cos \phi$, then $u+v = 2 t$ and 
\begin{align*}
\cos (\t \pm \phi) \,& = (t-s)(t+s) \mp \sqrt{1-(t-s)^2}\sqrt{1-(t+s)^2} \\
          & = t^2-s^2 \mp \sqrt{(1+t)^2 - s^2}\sqrt{(1- t)^2 - s^2},
\end{align*}
both of which which are even in $s$, so that $\sP_{j,n}^{(\a,\b)}(t-s,t+s)$ are polynomials in $(s,t)$ variables that are
even in the $s$ variable. Thus, they consist of an orthogonal basis for $\CV_n^\sE(\sW, \Diamond)$ with $n =2m$ or 
$2m+1$. In particular, it follows that $\sP_{j,n}^{(\a,\b)}(t-\|x\|,t+\|x\|)$ is a polynomial in $(x,t)$ variables. Replacing 
$\sW$ by $\sW^{(2k)}$, we obtain an orthogonal basis for the coupled cone. 

\begin{prop}
Let $\sP_{j,n}^{(\a,\b)}$ be defined as in \eqref{eq:squareOP1} and \eqref{eq:squareOP2}. Then an orthogonal basis 
for $\CV_n(\Wb_\diamond, \VV^{d+1})$ on the coupled cone is given by 
$$
    \Qb_{k,\ell}^n(x,t) = \sP_{j,n-k}^{(\a+ 2k + d-1, \b)}(t-\|x\|, t+\|x\|) Y_\ell^k(x),  \quad  
         0\le j \le \left \lfloor \frac{n-k}2\right \rfloor,\,\, 0 \le k \le n,
$$
where $\{Y_\ell^k: 0\le \ell \le \dim \CH_n^d\}$ is an orthogonal basis of $\CH_n^d$. 
\end{prop}
 
\subsection{Paraboloid}
We consider the parabolic domain $\Omega$, bounded by the line $t= 0$ and the parabola $t = 1 - s^2$, 
and the weight function defined by 
$$
   \Omega = \{(s,t): s^2 \le t,\, 0 \le t \le 1\}, \quad \hbox{and} \quad   \sW(s,t) =  |s|^{2\a} t^\b (t-s^2)^\g.
$$
Its rotation in the $t$ axis is the paraboloid defined by 
$$
     \VV^{d+1} =\left \{ (x,t): \|x\|^2 \le t, \, 0 \le t \le 1 \right\}
$$
and the weight function $\Wb(x)=\sW(\|x\|,t)$ becomes 
$$
  \Wb_{\a,\b,\g}(x,t) =  \|x\|^{2\a} (1-t)^\b (t - \|x\|^2)^\g, \qquad \b, \g > -1, \quad 2\a + \g > -d.
$$
The domain bounded by the parabola and $t=1$ and the paraboloid in 3D are depicted in Figure \ref{fig:Parabo}. 

\begin{figure}[htb]
\hfill\begin{minipage}[b]{0.44\textwidth} \centering
\includegraphics[width=0.9\textwidth]{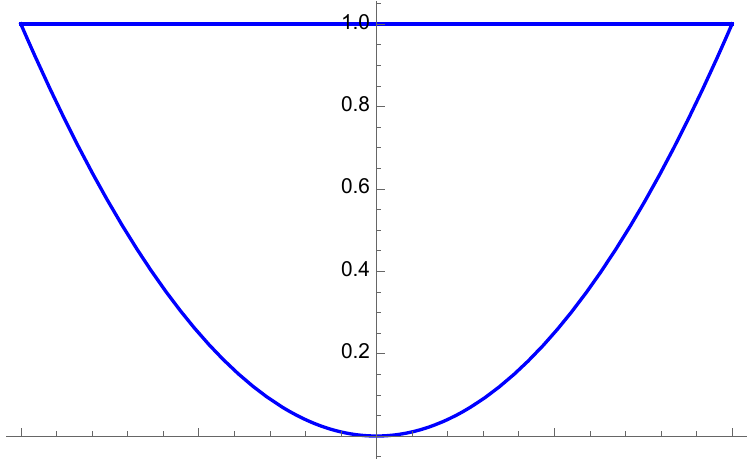}
\end{minipage}\hfill\begin{minipage}[b]{0.56\textwidth}
\centering
\includegraphics[width=0.82\textwidth]{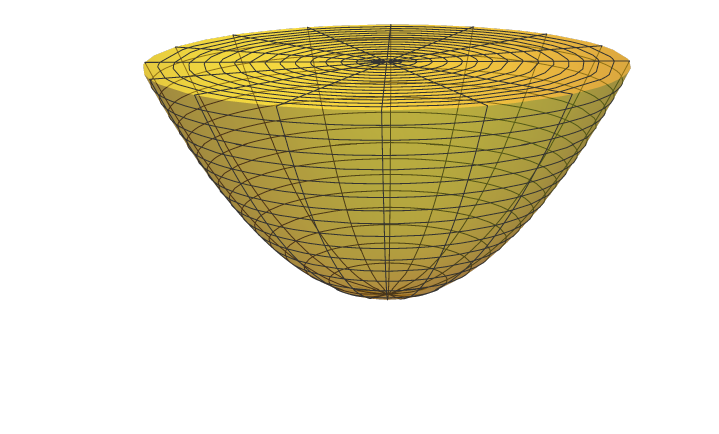}
\end{minipage}\hspace*{\fill}
\caption{Parabola and paraboloid}\label{fig:Parabo}
\end{figure}

Using \eqref{eq:intV=} and changing variables $s \mapsto \sqrt{s}$ and $t \mapsto 1-t$, it is easy to see that 
$$
  \int_{\VV^{d+1}} \Wb_{\a,\b,\g}(x,t)\d x \d t = \omega_d \f12 \int_{\triangle} s^{\a+d-1} t^\b (t-s)^\g\d s\d t = 
  \frac{\omega_d}{2\sb_{\a+\frac{d-1}{2},\b,\g}^\triangle},
$$
where $\omega_d$ denote the surface area of $\sph$ and $\sb_{\a,\b,\g}^\triangle$ is given in \eqref{eq:Jacobi-const}. 
In this case, it is easy to verify that an orthogonal basis for $\CV_m(\sW, \Omega)$ is given by 
$$
     \sP_j^{m}(s,t) = P_{m-j}^{(\a+\g+j+\f12,\b)}(1- 2 t) t^{\f j 2} C_j^{(\g+\f12,\a)}\left(\frac{s}{\sqrt{t}}\right), \quad 0 \le j \le m,
$$
in terms of the Jacobi polynomial $P_k^{(\a,\b)}$ and the generalized Gegenbauer polynomial $C_k^{(\l,\a)}$. For
$\a = 0$,  the polynomial $C_j^{(\l,0)} = C_j^\l$ is the ordinary Gegenbauer polynomial and this is the classical basis 
(\cite{K} and \cite[p. 40]{DX}) on the parabola domain $\Omega$. Since $C_j^{(\g,\a)}$ is an even polynomial if $j$ 
is even, by \eqref{eq:gGegen}, we obtain that an orthogonal basis for $\CV_m^{\sE}(\sW^{(2k)}, \Omega)$ consists of 
$$
 P_{m-2j}^{(2k+2j+ \a+\g+\f d 2,\b)}(1-2 t) t^j P_j^{(\a+2k+\f{d-2}2),\g}\left(1-\frac{2\|x\|^2}{t} \right), 
       \quad 0 \le j \le \left \lfloor \frac m 2\right \rfloor. 
$$ 
Consequently, we obtain an orthogonal basis for the paraboloid by Theorem \ref{thm:OP_V}. 

\begin{prop}
Let $\{Y_\ell^k: 1 \le \ell \le \dim \CH_k^d\}$ be an orthonormal basis of $\CH_k^d$. Define 
\begin{align} \label{eq:Q-paraboloid}
 \Qb_{j,k,\ell}^{n, (\a,\b,\g)}(x,t) & = P_{n-k-2j}^{(2j+ \a+\g+k+\f d 2,\b)}(1-2 t) t^j 
      P_j^{(\a+k+\f{d-2}2,\g)}\left(1- \frac{2\|x\|^2}{t}\right) Y_\ell^k(x) \notag\\
     & = \sT_{j, n-k-j}^{\a+k+\f{d-2}{2}, \b,\g}\left(\|x\|^2,1-t\right) Y_\ell^k(x) ,
\end{align}
where $\sT_{j,m-j}^{\a,\b,\g}$ is the polynomial on the triangle defined in \eqref{eq:triOP-T}. Then the polynomials in 
$
\left\{ \Qb_{j,k,\ell}^{n, (\a,\b,\g)}: 1 \le \ell \le \dim \CH_k^d,\, 0 \le j \le  \left \lfloor \tfrac{n-k}2\right \rfloor,\, 0\le k\le n \right\}
$ 
forms an orthogonal basis for $\CV_n(\Wb_{\a,\b,\g}, \VV^{d+1})$ on the paraboloid. Moreover,
$$
    \Hb_{j,k,n}^{\a,\b,\g} := \left \langle\Qb_{j,k,\ell}^{n, (\a,\b,\g)},\Qb_{j,k,\ell}^{n, (\a,\b,\g)} \right \rangle_\Wb
         = \sh_{j, n-k-j}^{\a+k+\f{d-2}{2}, \b,\g}
$$
in terms of the normalization constant in \eqref{eq:hjm-triangle}. 
\end{prop}

\begin{proof}
The first equality of \eqref{eq:Q-paraboloid} follows immediately from Theorem \ref{thm:OP_V}, whereas the second
one follows by comparing with \eqref{eq:triOP-T}. The norm of $\Qb_{j,k,\ell}^{n, (\a,\b,\g)}$ is computed by using
\eqref{eq:intV=} and changing variables $s \mapsto \sqrt{s}$ and $t \mapsto 1-t$, so that 
\begin{align*}
  \Hb_{j,k,n}^{\a,\b,\g},&= b_{\a,\b,\g}^\VV \int_{\VV^{d+1}} \left |\Qb_{j,k,\ell}^{n,(\a,\b,\g)}(x,t)\right|^2 W_{\a,\b,\g}(x,t) \d x \d t  \\
     &= \f{\o_d}2 b_{\a,\b,\g}^\VV \int_{\triangle} \left |\sT_{j, n-k-j}^{\a+k+\f{d-2}{2}, \b,\g}\left(s, t\right)\right|^2 
        w_{\a+k+\frac{d-2}{2},\b,\g}(s,t) \d s \d t\\
       & = \sh_{j, n-k-j}^{\a+k+\f{d-2}{2}, \b,\g},
\end{align*}
by the definition of  \eqref{eq:hjm-triangle}. \end{proof}

Two remarks are in order. First, for $\a = 0$, we can rewrite the basis in \eqref{eq:Q-paraboloid}, using 
$P_j^{(\a,\g)}(x) = (-1)^j P_j^{(\g,\a)}(-x)$ and setting $k = m - 2j$, 
$$
   \Qb_{j, m-2j,\ell}^{n,(0,\b,\g)}(x,t) = (-1)^j P_{n-m}^{(\g+m+\f d 2,\b)}(1-2 t) t^j P_{j,\ell}^m \left(W_{\g+\f12}; \frac{x}{t}\right) \notag
$$
in terms of $P_j^m(W_\mu)$ in \eqref{eq:basisBd}, the classical OPs on the unit ball. In this case, we can decompose 
the space $\CV_n(\Wb, \VV^{d+1})$ as a direct sum according to the spaces of spherical harmonics, so that each of the
subspace in the sum is an eigenspace of a second-order differential operator; in other words, the eigenvalues depend on 
two parameters, instead of just the degree of polynomials \cite[Proposition 3.2]{X23b}. Moreover, we have an addition formula 
that involves OPs of two variables on a parabolic domain instead of $Z_n^\lambda$ \cite[Theorem 3.3]{X23b}. These 
properties are utilized for analysis on the paraboloid in \cite{X23b}. 

Second, the basis \eqref{eq:Q-paraboloid} is derived from the composition of OPs on the right triangle, $\sT_{j,m}(s^2,t)$, 
with $s$ replaced by $\|x\|$ and the domain $\VV^{d+1}$ comes from revolving a region bounded by the right triangle. 
This raises the question of a possible extension using OPs on other triangles. We need a basis $\sP_j^m$ for the space 
$\CV_n(\sW^{(2k)}, \Omega)$ for $0 \le k \le n$, which requires the triangle to have one leg on $s= 0$ if $\sP_j^m$ can 
be given in terms of classical OPs of one variable. Furthermore, our construction requires $\deg \sP_j^m (u^2,v) = m+j$
in order to have an OP basis for $\CV_n^\sE(\sW, \Omega)$. Together, however, these requirements turn out to be 
rather restrictive. For example, they hold for $\sT_{j,m-j}$ in \eqref{eq:triOP-T} but not for $\sS_{j,m-j}$ in \eqref{eq:triOP-S} 
and $\sR_{j,m-j}$ in \eqref{eq:triOP-R}. A careful check of all possible cases shows that we obtain the desired basis of OPs 
only when the triangle, with one leg on $s=0$, is a right triangle. In other words, the only parabolic domain of revolution
that has OPs expressible by classical OPs of one variable and spherical harmonics appears to be essentially the paraboloid 
discussed in this subsection. As we shall see in the next two sections, the fully symmetric domains offer more possibilities. 

\section{Double domains of revolution}
\setcounter{equation}{0}
In this section, we assume that the domain $\Omega$ is symmetric in both $s$ and $t$ variables and give a construction 
of bases for $\CV_n^\sE(\Omega, \sW)$ by making use of the symmetry. This construction, discussed in the first 
subsection, is more flexible and leads to several new examples that illustrate the advantage of our new approach
developed in the previous section. After a brief second subsection on reproducing kernels, we present our examples 
based on parallelograms in the third subsection. Further examples will be given in the next section. 

\subsection{Orthogonal structure on fully symmetric domains}

We require both the domain $\Omega$ and the weight function $\sW$ to be fully symmetric. 

\begin{defn}
A domain $\Omega$ in $\RR^2$ is called fully symmetric if $(s,t) \in \Omega$ implies $(\pm s, \pm t) \in \Omega$.
A weight function $\sW$ on $\Omega$ is called fully symmetric if $\sW(s,t) = \sW(\pm s, \pm t)$ for all $(s,t)$ in
a fully symmetric domain $\Omega$. 
\end{defn}

A fully symmetric $\Omega$ is determined by its portion in the positive quadrant, which we denote by 
$$
\Omega_{+ ,+} := \{(u,v) \in \Omega: u \ge 0, v \ge 0\}.
$$
For a fully symmetric weight function $\sW$ and its domain $\Omega$, we further denote 
$$
      \sqrt{\Omega}:= \left \{\left(\sqrt{u},\sqrt{v}\right): (u,v) \in \Omega_{+,+} \right \}
           \quad \hbox{and}\quad \sW_{\pm \f12, \pm \f12}(u,v): = u^{\pm \f12} v^{\pm \f12} \sW\left(\sqrt{u},\sqrt{v}\right).
$$
We now show that an orthogonal basis for $\CV_n(\sW,\Omega)$ can be derived from four families of orthogonal bases
with respect to the inner products 
$$
  \la f,g\ra_{\pm \f12,\pm\f12} = \int_{\sqrt{\Omega}} f(u,v) g(u,v) \sW_{\pm \f12,\pm\f12} \left(u,v\right) \d u \d v.  
$$

\begin{thm} \label{thm:FullSym}
Let $\Omega$ and $\sW$ be fully symmetric. Let $\{\sP_{j,m}\big(\sW_{\pm \f12, \pm \f12}\big): 0 \le j \le m\}$ 
be an orthonormal basis of $\CV_n\big(\sW_{\pm \f12, \pm \f12}, \sqrt{\Omega}\big)$. Define
\begin{align}\label{eq:full-sym-e}
   \sP_j^{n} (\sW; u,v)  = \begin{cases} \sP_{j,m}\big(\sW_{-\f12, - \f12}; u^2,v^2\big), \quad 0 \le j \le m, & n = 2m, \\ 
                  v \,\! \sP_{j,m}\big(\sW_{-\f12, \f12}; u^2,v^2\big), \quad 0 \le j \le m, & n = 2m +1.\end{cases} 
 \end{align}
Then $\{\sP_j^n(\sW): 0 \le j \le \left \lfloor \frac n 2 \right \rfloor\}$ is an orthonormal basis for $\CV_n^\sE(\sW, \Omega)$. 
Furthermore, define
 \begin{align}\label{eq:full-sym-o}
     \sQ_j^{n} (\sW;u,v) = \begin{cases} u v \sP_{j,m-1}\big(\sW_{\f12, \f12}; u^2,v^2\big), \quad 0 \le j \le m-1, & n = 2m, \\ 
               u\,\! \sP_{j,m}\big(\sW_{\f12, -\f12}; u^2,v^2\big), \quad 0 \le j \le m, & n = 2m +1.\end{cases} 
\end{align}
Then $\{\sQ_j^n(W): 0 \le j \le \left \lfloor \frac n 2 \right \rfloor\}$ is an orthonormal basis for $\CV_n^\sO(\sW, \Omega)$. 
\end{thm} 

\begin{proof}
By definition, $\sP_j^n$ is symmetric in the $u$ variable. Moreover, since $\sP_j^{2m}$ is even in the $v$ variable and 
$\sP_j^{2m+1}$ is odd in $v$ variable, they are orthogonal with respect to the fully symmetric weight $\sW$ on $\Omega$.
Changing variables $u = s^2$ and $v = t^2$, it follows readily that
\begin{align*}
  \la \sP_j^{2m}, \sP_{j'}^{2m'} \ra_\Omega\, &  = 4 \int_{\Omega_{+,+}} \sP_j^{2m}(\sW;u,v) \sP_{j'}^{2m'}(\sW; u,v) \sW(u,v) \d u \d v \\
  &  = 4 \int_{\Omega_{+,+}}\sP_{j,m}\big(\sW_{-\f12, - \f12}; u^2,v^2\big)\sP_{j',m'}\big(\sW_{-\f12, - \f12}; u^2,v^2\big) \sW(u,v) \d u \d v\\
  &  = \int_{\sqrt{\Omega}} \sP_{j,m}\big(\sW_{-\f12, - \f12}; s,t\big) \sP_{j',m'}\big(\sW_{-\f12, - \f12}; s,t\big)\sW_{-\f12, -\f12}(s,t) \d s \d t \\
  & = \delta_{j,j'}\delta_{m,m'}.
\end{align*}
The same proof also shows the orthogonality of $\sP_j^{2m+1}$ and $\sP_{j'}^{2m'+1}$. Furthermore, the polynomials 
$\sQ_j^n$ are odd in the $u$ variable, so that they are orthogonal to $\sP_j^n$ by symmetry. The above proof can also 
be used to establish the orthogonality of $\sQ_j^n$ and $\sQ_{j'}^{n'}$. It is easy to see that the cardinality of 
$\{\sP_j^n, \sQ_j^n\}$ is exactly $n+1$ so that the set consists of an orthonormal basis of $\CV_n(\sW, \Omega)$. 
By parity, this shows that $\{\sP_j^n\}$ is an orthonormal basis for $\CV_n^\sE(\sW; \Omega)$ and $\{\sQ_j^n\}$ is an 
orthonormal basis for $\CV_n^\sO(\sW, \Omega)$.
\end{proof}

We now use the above theorem to construct an orthogonal basis on the domain of revolution, which will be symmetric
in the $t$ variable, and we shall denote it by $\XX^{d+1}$ instead of $\VV^{d+1}$ accordingly, 
$$
  \XX^{d+1} = \left \{(x,t) \in \RR^{d+1}, \quad x \in \RR^d,\, t \in \RR, \quad (\|x\|, |t|) \in \sqrt{\Omega} \right\}.
$$
To obtain a basis for $\CV_n(\Wb, \XX^{d+1})$, we need an orthogonal basis for $\CV_n^\sE(\sW^{(2k)}, \Omega)$ 
by Theorem \ref{thm:OP_V}. Recall the definition of $\sW^{(2k)}$ given in \eqref{eq:sWk}. We define 
\begin{equation} \label{eq:Wk_fullsym}
\sW_{-\f12,\pm \f12}^{(k)}(s,t) = |s|^{k+\frac{d-1}{2}} \sW_{-\f12, \pm \f12} \left(\sqrt{s}, \sqrt{t}\right), 
\quad (s,t) \in \sqrt{\Omega}.
\end{equation}
Then the polynomials in \eqref{eq:full-sym-e} consist of an orthogonal basis for $\CV_n^E(\sW^{(2k)}, \Omega)$ if 
we replace $\sW_{-\f12, \pm \f12}$ by $\sW_{-\f12, \pm \f12}^{(k)}$. Consequently, by Theorem \ref{thm:OP_V}, we 
obtain the following corollary.

\begin{cor} \label{cor:FullSym}
Let $\Omega$ and $\sW$ be fully symmetric. Let $\{\sP_{j,m}\big(\sW^{(k)}_{\pm \f12, \pm \f12}\big): 0 \le j \le m\}$ 
be an orthonormal basis of $\CV_n\big(\sW^{(k)}_{\pm \f12, \pm \f12}, \sqrt{\Omega}\big)$. Define 
$$
\sP_j^{n-k} (\sW^{(2k)}; s, t) = \begin{cases} \sP_{j,m}\left(\sW_{-\f12, - \f12}^{(k)}; s^2,t^2\right), 
                      \,\, 0 \le j \le m, & n-k = 2m, \\ 
                  t \,\! \sP_{j,m}\left(\sW_{-\f12, \f12}^{(k)}; s^2,t^2\right), \,\, 0 \le j \le m, & n-k = 2m +1.\end{cases}
$$
Then $\{\sP_j^m\big(\sW^{(2k)}\big): 0 \le j \le \left \lfloor \frac m 2 \right \rfloor\}$ is an orthonormal basis for 
$\CV_m^\sE(\sW^{(2k)}, \Omega)$ and 
\begin{align} \label{eq:OP_fullsym}
          \Qb_{j,k,\ell}^{n} (x,t) = \sP_j^{n-k} (\sW^{(2k)}; \|x\|, t) Y_\ell^k(x),\quad 0\le j \le \left \lfloor \frac{n-k}{2} \right \rfloor, \, 0 \le k \le n,
\end{align}
where $\{Y_\ell^k: 1 \le \ell \le \dim \CH_k^d\}$ is an orthonormal basis for $\CH_k^d$, consist of an orthogonal basis for the
space $\CV_n(\Wb,\XX^{d+1})$.
\end{cor}

The space of OPs $\CV_n(\Wb,\XX^{d+1})$ has a natural decomposition in terms of the parity of the polynomials 
in the $t$ variable. 

\begin{defn}
Let $\Wb$ be an even weight function in $t$. We denote by $\CV_n^\sE(\Wb, \XX^{d+1})$ the subspace of 
$\CV_n(\Wb,\XX^{d+1})$ that consists of polynomials even in $t$ variable. Similarly, 
$\CV_n^\sO(\Wb,\XX^{d+1})$ denotes the subspace that consists of polynomials odd in $t$ variable. 
\footnote{Let us emphasis that $\CV_n^\sE(\sW, \Omega)$ is the space of OPs on $\Omega \subset \RR^2$ 
that are even in the \underline{$s$, or the first}, variable, whereas $\CV_n^\sE(\Wb, \XX^{d+1})$ is the space of OPs on
$\XX^{d+1}$ that are even in the \underline{$t$, or the last}, variable. }
\end{defn}
By the definition, it follows immediately that 
\begin{equation} \label{eq:CVsplit}
  \CV_n\left(\Wb,\XX^{d+1}\right) = \CV_n^\sE\left(\Wb,\XX^{d+1}\right) \bigoplus \CV_n^\sO\left(\Wb,\XX^{d+1}\right).
\end{equation}
The polynomials $\Qb_{j,n-2m,\ell}^{n}$ in \eqref{eq:OP_fullsym} consist an orthonormal basis for  $\CV_n^\sE(\Wb,\XX^{d+1})$ 
and those polynomials $\Qb_{j,n-2m-1,\ell}^{n}$ in \eqref{eq:OP_fullsym} consist an orthonormal basis for
$\CV_n^\sE(\Wb,\XX^{d+1})$. In particular, we conclude that 
\begin{align}
  \dim \CV_n^\sE(\Wb,\XX^{d+1}) = &
     \sum_{m=0}^{\lfloor \frac{n}{2} \rfloor} (m+1) \dim \CH_{n-2m}^d = \sum_{m=0}^{\lfloor \frac{n}{2} \rfloor} \binom{n-2m+d-1}{d-1}.
      \label{eq:dimE} \\
   \dim \CV_n^\sO(\Wb,\XX^{d+1}) = & \sum_{m=0}^{\lfloor \frac{n-1}{2} \rfloor} (m+1) \dim \CH_{n-2m-1}^d 
    = \sum_{m=0}^{\lfloor \frac{n-1}{2} \rfloor} \binom{n-2m+d-2}{d-1}.     \label{eq:dimO}
\end{align}
Moreover, as a consequence of Corollary \ref{cor:FullSym}, orthogonal bases for these spaces can be derived from 
orthogonal bases on $\sqrt{\Omega}$, which we formulate as a proposition for easier reference. 

\begin{prop} \label{prop:OP_fullSym}
Let $\big\{\sP_{j,m}\big(\sW^{(k)}_{-\f12, \pm \f12}\big): 0 \le j \le m\big\}$ be an orthonormal basis of 
$\CV_n\big(\sW^{(k)}_{- \f12, \pm \f12}, \sqrt{\Omega}\big)$ and let $\{Y_\ell^k: 1 \le \ell \le \dim \CH_k^d\}$ be 
an orthonormal basis for $\CH_k^d$. Then the space $\CV_n^\sE\big(\Wb,\XX^{d+1}\big)$ has 
an orthonormal basis given by 
\begin{equation} \label{eq:basisE}
 \Qb_{j,n-2m,\ell}^n(x,t) =  \sP_{j,m} \Big(\sW_{-\f12, -\f12}^{(n-2m)}; \|x\|^2, t^2\Big) Y_\ell^{n-2m} (x)
\end{equation}
for $1 \le \ell \le \dim \CH_{n-2m}^d, \,0\le j \le  m \le \lfloor \frac{n}{2} \rfloor$, and the space $\CV_n^\sO(W,\XX^{d+1})$
has an orthonormal basis given by 
\begin{equation} \label{eq:basisO}
 \Qb_{j,n-2m-1,\ell}^n(x,t) =  \sP_{j,m} \Big(\sW_{-\f12, \f12}^{(n-2m-1)}; \|x\|^2, t^2\Big) Y_\ell^{n-2m-1} (x) 
\end{equation}
for $1 \le \ell \le \dim \CH_{n-2m-1}^d, \,0\le j \le m \le \lfloor \frac{n-1}{2} \rfloor$. 
\end{prop}

Two remarks are in order. First, for our goal of constructing explicit orthogonal bases on the domain $\XX^{d+1}$, we 
require our weight function contains the factor $|s|^{2k+d+1}$ in $\sW_{-\f12, \pm \f12}^{(k)}$, for which we need the 
line $s=0$ to be part of the boundary of $\sqrt{\Omega}$. This, however, appears to be the only constraint for the fully 
symmetric domain and weight. Consequently, there are plenty of examples of various domains of revolutions, for
which explicit orthogonal basis can be written down. 
Second, it is often more convenient to start with $\sqrt{\Omega}$ and the weight function $\sw(u,v)$ defined on
$\sqrt{\Omega}$, so that the weight function $\sW$ on $\Omega$ and $\Wb$ on $\XX^{d+1}$ become
\begin{equation}\label{eq:w-varpi}
   \sW(s,t) = \sw\left(s^2,t^2\right) \quad \hbox{and} \quad \Wb(x,t) = \sw\left(\|x\|^2,t^2\right)
\end{equation}
and, more conveniently, $\sW_{\pm \f12, \pm \f12}$ and $\sW_{-\f12,\pm \f12}^{(k)}$ in \eqref{eq:Wk_fullsym} become
\begin{equation}\label{eq:w-varpi-k}
   \sW_{\pm \f12, \pm \f12}(s,t) = s^{\pm \f12} t^{\pm \f12} \sw(s,t) \quad \hbox{and}\quad
    \sW_{-\f12,\pm \f12}^{(k)}(s,t) =  |s|^{k+\frac{d-2}{2}} t^{\pm \f12} \sw(s,t).  
\end{equation}
In the following, we shall adopt this convention when discussing our examples. 

\subsection{Reproducing kernels and Fourier orthogonal series} \label{set:4.2}
The reproducing kernel of $\CV_n(\Wb, \XX^{d+1})$ is denoted by $\Pb_n(\Wb;\cdot,\cdot)$. For the fully symmetric 
domain, by \eqref{eq:CVsplit}, we can write 
$$
  \Pb_n\big(\Wb; (x,t), (y,s)\big) = \Pb_n^\sE\big(\Wb; (x,t), (y,s)\big)+\Pb_n^\sO \big(\Wb; (x,t), (y,s)\big),
$$
where $\Pb_n^\sE (\Wb)$ and $\Pb_n^\sO (\Wb)$ denote the the reproducing kernels for $\CV_n^\sE(\Wb, \XX^{d+1})$ 
and $\CV_n^\sO(\Wb, \XX^{d+1})$, respectively. Recall that the projection operator $\proj_n (\Wb; f)$ is an integral operator that has $\Pb_n(\Wb;\cdot,\cdot)$ as its
kernel. For $f \in L^2(\Wb, \XX^{d+1})$, we define 
$$
  f^\sE(x,t) = \frac12 \left[ f(x,t) + f(x,-t)\right] \quad \hbox{and}\quad f^\sO(x,t) = \frac12 \left[ f(x,t) - f(x,-t)\right]. 
$$ 
Then $f(x,t) = f^\sE(x,t) + f^\sO(x,t)$, and $f^\sE$ is even in the $t$ variable and $f^\sO$ is odd in the $t$ variable. 
The following proposition is an immediate consequence of the parity and the orthogonality of the function and the kernel. 

\begin{prop}
For $f \in L^2(\Wb, \XX^{d+1})$, 
$$
   \proj_n(\Wb; f)(x,t) = \proj_n^\sE \left(\Wb; f^E\right) + \proj_n^\sO\left(\Wb; f^\sO\right),
$$
where, for $\sY = \sE$ or $\sY = \sO$, 
$$
  \proj_n^\sY(\Wb; f, x,t) = \int_{\XX^{d+1}} f(y,s) \Pb_n^\sY (\Wb; (x,t), (y,s)) \Wb(y,s) \d y \d s.
$$
\end{prop}

By its definition, $\proj_n^\sY$ is the projection operator $L^2(\Wb, \XX^{d+1}) \mapsto \CV_n^\sY(\Wb,\XX^{d+1})$. 
Hence, if $f$ is even in the $t$ variable, then it is easy to see that the Fourier orthogonal series of $f$ satisfies 
\begin{equation}\label{eq:Fourier-even}
   f = \sum_{n=0}^\infty \proj_n^\sE f. 
\end{equation}
This expansion could also be regarded as studying the Fourier orthogonal series of $f$ on the upper part of $\XX^{d+1}$, 
denote by 
$$
        \XX_+^{d+1} = \left\{(x,t) \in \XX^{d+1}: t \ge 0\right \},
$$
which is the rotation of $\Omega^+ = \{(s,t) \in \Omega: t \ge 0\}$ of the fully symmetric domain $\Omega$. Indeed, if 
$f\in L^2(\Wb, \XX_+^{d+1})$, then we can extend it to $\XX^{d+1}$ by defining $f(-x,t) = f(x,t)$. Let us call this extended
function $F$. Then $F$ is even in the $t$ variable so that it has the Fourier expansion \eqref{eq:Fourier-even}, which
gives the Fourier expansion of $f$ when restricted back to $\XX_+^{d+1}$.  
It should be noted, however, that the above Fourier expansion for $f$ is different from the Fourier orthogonal expansion
of $f$ in $L^2(\Wb, \XX_+^{d+1})$. For start, the dimension of $\CV_n^\sE(\Wb, \XX^{d+1})$ is different from the
dimension of $\CV_n(\Wb, \XX_+^{d+1})$. Nevertheless, as we shall show below, the orthogonal structure of some 
fully symmetric $\Wb$ and $\XX^{d+1}$ possess desirable properties that the structure of $\CV_n(\Wb, \XX_+^{d+1})$
may not have. 

By Proposition \ref{prop:OP_fullSym}, the orthogonal basis for $\CV_n^{\sE}(\Wb, \XX^{d+1})$ requires OPs of two variables
for the weight function $\sW_{-\f12, -\f12}^{(n-2m)}$, whereas the basis for $\CV_n^{\sO}(\Wb, \XX^{d+1})$ requires
OPs for the weight function $\sW_{-\f12, \f12}^{(n-2m)}$, which are different weight functions. Thus, it is not surprising 
that the spectral operator and addition formula for the two spaces could be different. Putting it another way, as shown in 
\cite{X21a} for the hyperboloid, we could say that the two properties hold only for $\CV_n^{\sE}(\Wb, \XX^{d+1})$ or 
for $\CV_n^{\sO}(\Wb, \XX^{d+1})$, but the same formulation may not hold for both. 

In Section 5, we will show that the two properties hold for $\CV_n^{\sE}(\Wb, \XX^{d+1})$ for some $\Wb$ and $\XX^{d+1}$. 
The above discussion shows that such a formulation provides powerful tools for studying the Fourier orthogonal series for 
functions either even or odd in the $t$ variable. 

\subsection{Cylindrical domains} \label{sect:square_radial}
As our first example, we consider the case when the domain $\sqrt{\Omega}$ is a parallelogram with one side on 
the axis $u = 0$. The trivial case is the rectangle $\sqrt{\Omega} = \{(u,v): 0 \le u \le \fa, \, 0 \le v \le \fb\}$, for which 
the corresponding $\XX^{d+1}$ will be a fully symmetric cylinder
$$
  \XX^{d+1} = \{(x,t): \|x\| \le \fa, \quad -\fb \le t \le  \fb\},
$$
which is the tensor product of $[-\fb,\fb] \times \BB^d_\fa$ (cf. \cite{DX,W}). If we deform the rectangle to a parallelogram, 
however, we end up with several distinct domains. In view of their geometry, we consider two cases. 

\subsubsection{Cylindrical domains caped by quadratic surfaces}
Let $\fa > \fb \ge  0$. We consider 
$$
  \sqrt{\Omega} = \{(u,v): 0 \le u \le 1, \,  \fb + (1-\fa)u \le v \le \fa+ (1-\fa)u\}.
$$
The rotation in the $t$ axis of the fully symmetric domain $\Omega =\{(s,t): (s^2,t^2) \in \sqrt{\Omega}\}$ leads to 
$$
  \XX^{d+1} =\left \{(x,t): \fb + (1-\fa) \|x\|^2 \le t^2 \le  \fa + (1-\fa) \fa \right\}.
$$
For $\fb > 0$, the domain has two parts that mirror each other. The upper part is bounded by a cylinder with its two ends caped by 
hyperbolic surfaces $\fb + (1-\fa)\|x\|^2 = t^2$ and $\fa + (1-\fa)\|x\|^2 = t^2$, whereas for $\fb = 0$, the lower cap 
of the cylinder is the cone $(1-\fa)\|x\|^2 = t^2$. For $\fa = \f58$, $\fb = \f1{16}$, the domain $\XX^{d+1}$ is depicted 
in the right-hand side of Figure \ref{fig:two-ellipsoids1}, 
\begin{figure}[htb]
\hfill
\begin{minipage}[b]{0.44\textwidth} \centering
\includegraphics[width=1.10\textwidth]{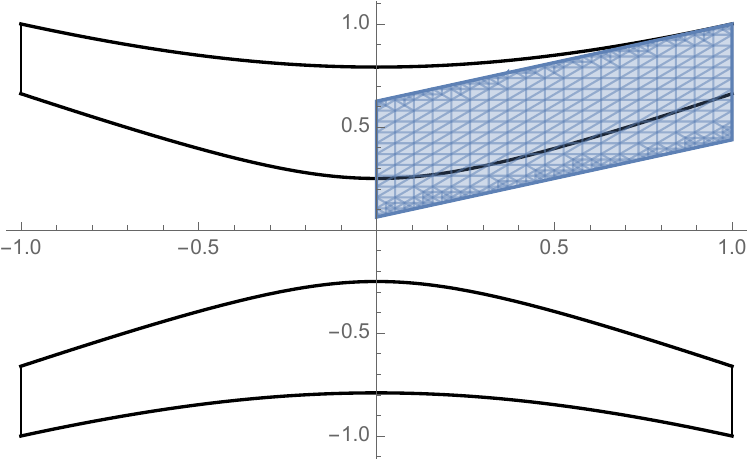}
\end{minipage}\hfill
\begin{minipage}[b]{0.54\textwidth}
\centering
\includegraphics[width=0.6\textwidth]{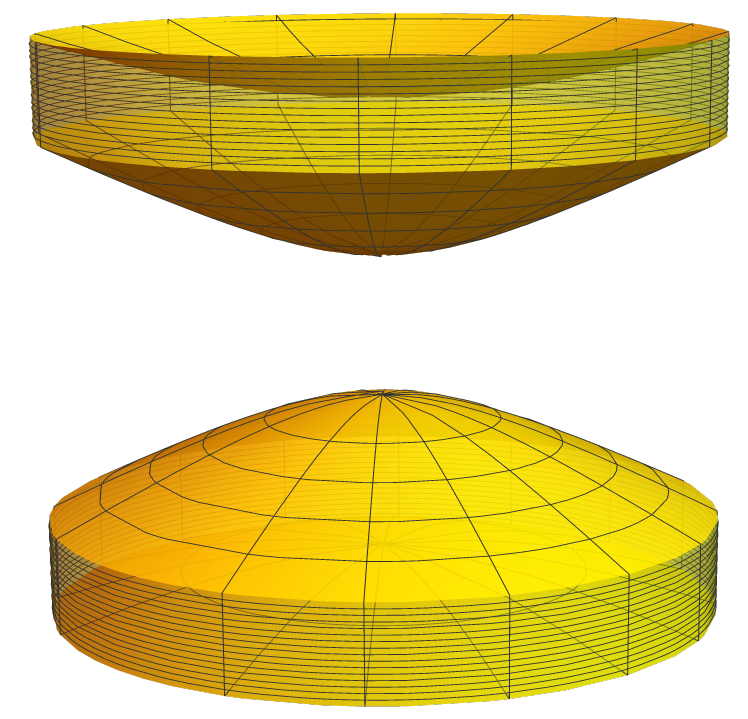}
\end{minipage}
\hspace*{\fill}
\caption{Cylindrical domains caped by quadratic surfaces: $\fa = \f58$ and $\fb= \f1{16}$} \label{fig:two-ellipsoids1}
\end{figure}
and $\Omega$ is depicted in the left-hand side with $\sqrt{\Omega}$ being the shaded parallelogram. 

For $\a, \b, \g,\t > -1$, let $\sw_{\a,\b,\g,\t}$ be the weight function defined by 
\begin{equation}\label{eq:w-2ellipse1}
  \sw_{\a,\b,\g,\t}(s,t) = s^\a (1- s)^\b (t- \fb - (1-\fa)s)^\g (\fa + (1-\fa)s - t)^{\t} t^{\f12}, \quad (s,t) \in \sqrt{\Omega}.
\end{equation}
The corresponding weigh function on $\XX^{d+1}$ is defined by 
\begin{align}\label{eq:W-2ellipse1}
   & \Wb_{\a,\b,\g,\t}(x,t) =   \sw_{\a,\b,\g,\t}\left(\|x\|^2,t^2\right) \\
    & \quad = \|x\|^{2\a} \left(1- \|x\|^2 \right)^\b \left(t^2- \fb - (1-\fa)\|x\|^2\right)^\g 
     \left(\fa + (1-\fa)\|s\|^2 - t^2\right)^{\t}|t|. \notag
\end{align}
By \eqref{eq:w-varpi-k}, the weight function $\sW^{(k)}_{-\f12,-\f12}$ becomes 
\begin{align*}
  \sW_{-\f12, -\f12}^{(k)}(s,t) \, &=  s^{k+\a + \f{d-2}{2}} (1- s)^\b (t- \fb - (1-\fa)s)^\g (\fa + (1-\fa)s - t)^{\t} \\
      & =(\fa-\fb)^{\g+\t} \varpi_{k+\a + \f{d-2}{2},\b}(s) \varpi_{\g,\t}\left( \frac{t-(1-\fa)s - \fb}{\fa-\fb}\right),
\end{align*}
where $\varpi_{a,b}(x) = x^a(1-x)^b$ is the Jacobi weight on the interval $[0,1]$. Consequently, an orthogonal basis for
the space $\CV_m\big(\sW^{(k)}_{- \f12, -\f12}, \sqrt{\Omega}\big)$ can be given in terms of the
Jacobi polynomials: for $0 \le j \le m$, 
$$
  \sP_{j,m}^{(k)}(s,t) = P_j^{(\a+k+\f{d-1}{2}, \b)}(1-2 s) P_{m-j}^{(\g,\t)} \left(1 - 2 \frac{t- (1-\fa) s - \fb}{\fa-\fb}\right).
$$

By Proposition \ref{prop:OP_fullSym}, we obtain immediately that the space $\CV_n^\sE(\Wb_{\a,\b,\g,\t},\XX^{d+1})$ 
has an orthogonal basis given by 
\begin{align} \label{eq:OP_paralle1}
\sQ_{j,n-2m,\ell}(x,t) = \, & P_j^{(\a+n-2m+\f{d-1}{2}, \b)}(1-2 \|x\|^2) \\ 
       &\times P_{m-j}^{(\g,\t)} \left(1 - 2 \frac{t^2- (1-\fa)\|x\|^2-\fb}{\fa-\fb}\right)
 Y_\ell^{n-2m}(x) \notag
\end{align}
with $1 \le \ell \le \dim \CH_{n-2m}^d$ and $0 \le j \le m \le \lfloor \frac n 2 \rfloor$. By  \eqref{eq:basisO}, we can also
derive a basis for $\CV_n^\sO(\Wb,\XX^{d+1})$ but it is for the weight function $\Wb(x,t) = \Wb_{\a,\b,\g,\t}(x,t) |t|^{-2}$,
differing from $\Wb_{\a,\b,\g,\t}(x,t)$ by an additional $|t|^{-2}$, since the basis in \eqref{eq:basisO} uses OPs for
$\sW_{-\f12, \f12}^{(k)}$, instead of $\sW_{-\f12, -\f12}^{(k)}$, which we need to choose so that it is again the product of
two Jacobi weight functions. 

\subsubsection{Cylindrical domains between two ellipsoid surfaces}
Let $\sa > \sb > 0$. We consider 
$$
  \sqrt{\Omega} = \{(u,v): 0 \le u \le 1, \,  \fb \le v+\fb u \le \fa\}.
$$
The rotation in the $t$ axis of the fully symmetric domain $\Omega =\{(s,t): (s^2,t^2) \in \sqrt{\Omega}\}$ leads to 
$$
  \XX^{d+1} = \{(x,t): \fb \le t^2+ \|x\|^2 \le \fa\},
$$
bounded by the surfaces of two ellipsoids: $t^2 + \|x\|^2 = \fa$ and $t^2 + \|x\|^2 = \fb$. For $\fa = 1$, $\fb = \f12$,
these domains are depicted in Figure \ref{fig:two-ellipsoids}, 
\begin{figure}[htb]
\hfill
\begin{minipage}[b]{0.44\textwidth} \centering
\includegraphics[width=0.80\textwidth]{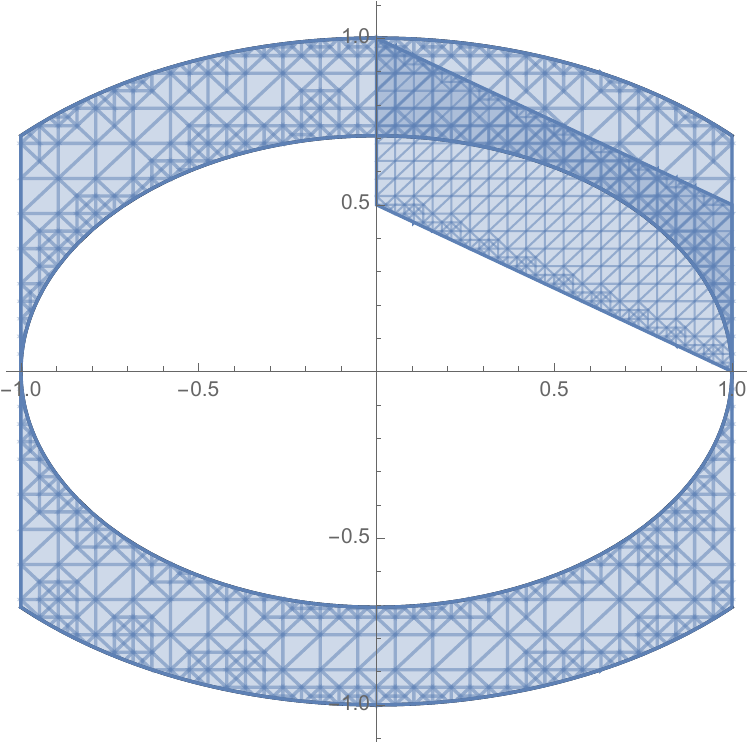}
\end{minipage}\hfill
\begin{minipage}[b]{0.56\textwidth}
\centering
\includegraphics[width=0.6\textwidth]{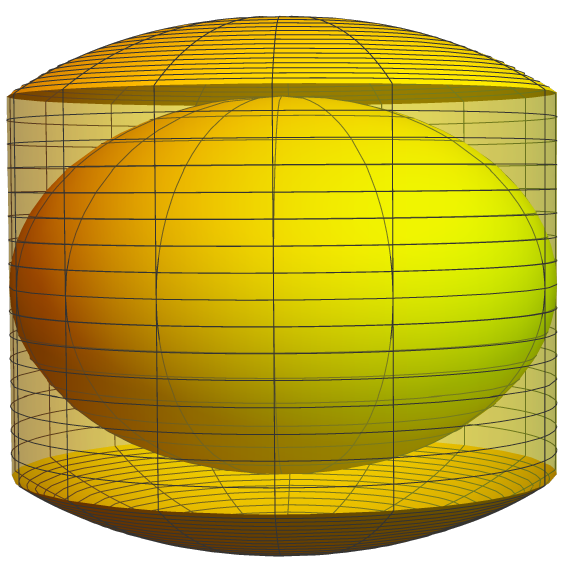}
\end{minipage}
\hspace*{\fill}
\caption{Cylindrical domains between two ellipsoid surfaces} \label{fig:two-ellipsoids}
\end{figure}
where $\Omega$ is the shaded domain between two ellipses and two vertical lines and $\sqrt{\Omega}$ is the 
shaded parallelogram in the left-hand side figure. 

For $\a,\b,\g, \t >  -1$, let $\sw_{\a,\b,\g,\t}$ be the weight function defined by
\begin{equation}\label{eq:w-2ellipse}
  \sw_{\a,\b,\g,\t}(s,t) = s^\a (1- s)^\b (t+ \fb s - \fb)^\g (\fa - t-\fb s)^{\t} t^{\f12}, \quad (s,t) \in \sqrt{\Omega}
\end{equation}
so that the weight function on $\XX^{d+1}$ is given by
\begin{align}\label{eq:W-2ellipse}
   \Wb_{\a,\b,\g,\t}(x,t) \, & =  \sw_{\a,\b,\g,\t}(\|x\|^2,t^2) \\
     & = \|x\|^{2\a} (1- \|x\|^2)^\b \left(t^2 +\fb \|x\|^2 - \fb\right)^\g \left(\fa - t^2 -\fb \|x\|^2\right)^\t |t|.  \notag
\end{align}
In this case, the weight function $\sW^{(k)}_{-\f12,-\f12}$ in \eqref{eq:w-varpi-k} becomes 
\begin{align*}
  \sW_{-\f12, -\f12}^{(k)}(s,t) \, &=  s^{k+\a + \f{d-2}{2}} (1- s)^\b (t+\fb s -\fb)^\g (\fa - t- \fb s)^{\t} \\
      & = (\fa-\fb)^{\g+\t} \varpi_{k+\a + \f{d-2}{2},\b}(s) \varpi_{\g,\t}\left( \frac{t + \fb s - \fb}{\fa-\fb}\right),
\end{align*}
where $\varpi_{a,b}$ is again the Jacobi weigh on $[0,1]$. Thus, an orthogonal basis for the space 
$\CV_m\big(\sW^{(k)}_{- \f12, -\f12}, \sqrt{\Omega}\big)$ can be given in terms of the Jacobi polynomials as
$$
  \sP_{j,m}^{(k)}(s,t) = P_j^{(\a+k+\f{d-1}{2}, \b)}(1-2 s) P_{m-j}^{(\g,\t)} \left(1 - 2 \frac{t + \fb s - \fb}{\fa-\fb}\right).
$$
Consequently, by Proposition \ref{prop:OP_fullSym}, the space $\CV_n^\sE(\Wb_{\a,\b,\g,\t},\XX^{d+1})$ 
has an orthogonal basis given by 
\begin{align} \label{eq:OP_paralle2}
\sQ_{j,n-2m,\ell}(x,t) = \, & P_j^{(\a+n-2m+\f{d-1}{2}, \b)}(1-2 \|x\|^2) \\ 
       &\times P_{m-j}^{(\g,\t)} \left(1 - 2 \frac{t^2- (-1-\fa)\|x\|^2-\fb}{\fa-\fb}\right) Y_\ell^{n-2m}(x), \notag
\end{align}
for $1 \le \ell \le \dim \CH_{n-2m}^d$ and $0 \le j \le m \le \lfloor \frac n 2 \rfloor$. As in the previous case, we
can also give an orthogonal basis for $\CV_n^\sO(\Wb,\XX^{d+1})$ but for the weight function 
$\Wb(x,t) = \Wb_{\a,\b,\g,\t}(x,t) |t|^{-2}$.

\section{Double conic and hyperbolic domains}
\setcounter{equation}{0}

This section contains several new double domains of revolution on which an orthogonal basis can be given 
explicitly. By assuming that the weight function is even in the $t$ variable, we can divide OPs into two
families based on their parity in the $t$ variable. For our examples, the two essential properties, spectral 
operator and addition formula, hold for orthogonal spaces that consist of OPs even in the $t$ variable. This 
was established in \cite{X21a} for the double cone, which will be reexamined in the first subsection, for which 
$\sqrt{\Omega}$ is a right triangle. The new examples in follow-up subsections correspond to triangles of 
different types. 

\subsection{Double cone} \label{sect:double_cone}
We consider the case that the domain $\sqrt{\Omega}$ is the right triangle with vertices $(0,0)$, $(0,1)$ and $(1,1)$; that is, 
\begin{equation}\label{eq:tri-nabla}
  \nabla:= \sqrt{\Omega} = \left \{(u,v):  0 \le  u \le v  \le 1 \right\}. 
\end{equation}
The rotation in the $t$-axis of the fully symmetric domain $\Omega =\{(s,t): (s^2,t^2) \in \sqrt{\Omega}\}$ 
leads to the double cone, again denoted by $\XX^{d+1}$ instead of $\VV^{d+1}$, bounded by the double conic 
surface and hyperplanes $t = \pm 1$ at its two ends, which is depicted on the right-hand side of Figure \ref{fig:Double_Cone}. 
\begin{figure}[htb]
\hfill
\begin{minipage}[b]{0.44\textwidth} \centering
\includegraphics[width=0.88\textwidth]{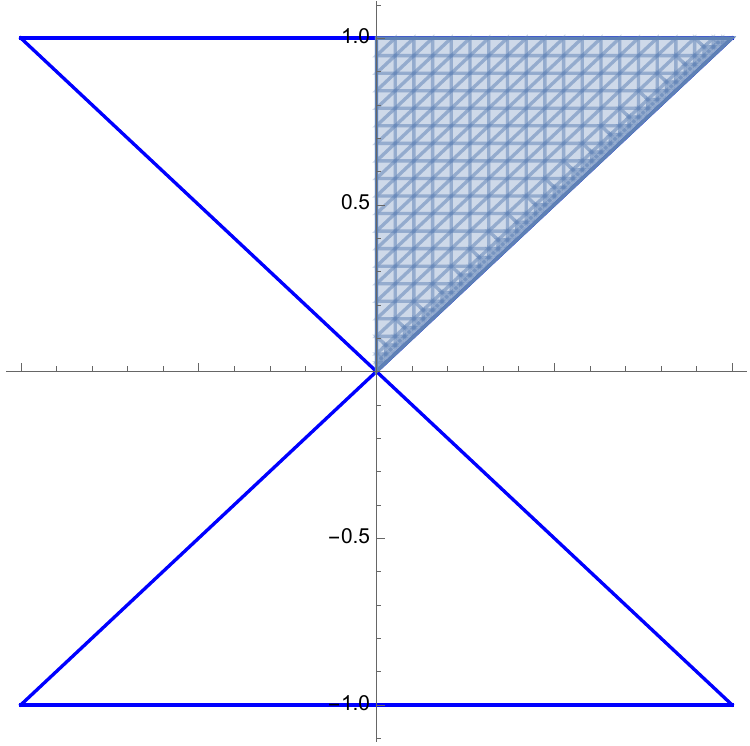}
\end{minipage}\hfill
\begin{minipage}[b]{0.52\textwidth}
\centering
\includegraphics[width=0.6\textwidth]{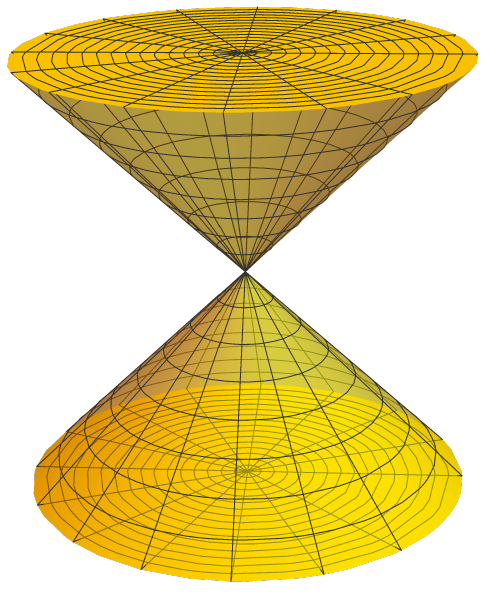}
\end{minipage}
\hspace*{\fill}
\caption{Double cone} \label{fig:Double_Cone}
\end{figure}
The domain $\Omega$ is depicted on the left-hand side of Figure \ref{fig:Double_Cone} and $\sqrt{\Omega}$ is the shaded 
triangle. A family of OPs on the double cone was studied in \cite{X21a}. Below we discuss it under our setup 
for a family of weight functions more general than the one in \cite{X21a}. 

\subsubsection{OPs on double cone}
For $\a,\b,\g, \t >  -1$, let $\sw_{\a,\b,\g,\t}$ be the weight function defined by
\begin{equation}\label{eq:w-4parameters}
  \sw_{\a,\b,\g,\t}(s,t) = s^\a (1- t)^\b (t- s)^\g t^{\t}, \quad (s,t) \in \nabla,
\end{equation}
so that the weight function on $\XX^{d+1}$ is given by
\begin{equation}\label{eq:W-4paraHyp}
   \Wb_{\a,\b,\g,\t}(x,t) =  \sw_{\a,\b,\g,\t}(\|x\|^2,t^2) = \|x\|^{2\a} (1- t^2)^\b (t^2- \|x\|^2)^\g |t|^{2 \t}. 
\end{equation}
 
In terms of the four-parameter Jacobi weight in \eqref{eq:Jacobi-w+} and follow the notation \eqref{eq:w-varpi-k},
\begin{equation} \label{eq:sW-cone}
\sW_{- \f12, \pm \f12}^{(k)}(s,t) = \sW_{k+\a+\frac{d-2}2, \b, \g, \t\pm \f12}(s, 1-t).
\end{equation}
Thus, by Theorem \ref{thm:FullSym}, the orthogonal basis \eqref{eq:full-sym-e} for the space $\CV_m(\sW^\sE{(2k)}, \Omega)$ 
becomes 
\begin{align} \label{eq:pre_basis_cone1}
\begin{cases}
\sT_{j,m}^{k+\a+\frac{d-2}2, \b, \g, \t -\f12}(s^2,1-t^2), & n = 2m, \\
 t \sT_{j,m}^{k+\a+\frac{d-2}2, \b, \g, \t + \f12}(s^2,1-t^2), & n = 2m+1, 
\end{cases} \quad 0 \le j \le m,
\end{align}
in terms of the OPs \eqref{eq:triOP-T2} on the triangle. Hence, by Corollary \ref{cor:FullSym}, we obtain an
orthogonal basis on the double cone. 

\begin{prop} \label{prop:Va_a=1A}
Let $\b,\g > -1$, $\a > -\f d2$, and $\a+\g+\t \ge -\f d2$. Let $\{Y_\ell^{k-2j}\}$ be an orthonormal basis of $\CH_n^d$. Then
the space $\CV_n(\Wb_{\a,\b,\g,\t}, \XX^{d+1})$ has an orthogonal basis that consists of 
\begin{equation}\label{eq:basisConeA}
\Qb_{j,k,\ell}^n (x,t) =  Y_\ell^k (x) \begin{cases}
  \sT_{j,m}^{k+\a+\frac{d-2}2, \b, \g, \t -\f12}(\|x\|^2,1-t^2), & n-k = 2m, \\
 t \sT_{j,m}^{k+\a+\frac{d-2}2, \b, \g, \t + \f12}(\|x\|^2,1-t^2), & n-k = 2m+1 
\end{cases}
\end{equation}
for $1 \le \ell \le \dim \CH_{k-2j}^d$ and $0 \le j \le m$.
\end{prop}

The basis in \eqref{eq:basisConeA} can be written in an alternative form, but it is useful in its present form for 
further examples to be discussed in the next section. For the alternative form, we need to use the generalized 
Gegenbauer polynomials $C_n^{(\l,\mu)}$ and the Jacobi polynomials $P_n^{(\a,\b)}$, which we state as a lemma. 

\begin{lem}\label{lem:sWk_a=1}
An orthogonal basis for $\CV_n^\sE(\sW^{(2k)}, \Omega)$ consists of polynomials
\begin{align*} 
 \sP_j^n(\sW^{(2k)}; s,t)= C_{n-2j}^{(\b+ \f12, 2 j+k+\a+\g+\t+\f {d} 2)}(t) t^{2j} P_j^{(\g, k+\a+\frac{d-2}2)}\left(\frac{2s^2}{t^2}-1\right), 
 \quad 0 \le j \le n/2.
\end{align*}
\end{lem}

\begin{proof} 
By \eqref{eq:triOP-T2}, the basis \eqref{eq:pre_basis_cone1} can be given in terms of the Jacobi polynomials
\begin{align*}
\begin{cases}
P_{m-j}^{(2j+k+\a+\g+\t+\frac{d-1}2,\b)}(1- 2 t^2)t^{2j} P_j^{(k+\a+\frac{d-2}2,\g)}\left(1- \frac{2s^2}{t^2}\right), & n = 2m, \\
 t P_{m-j}^{(2j+k+\a+\g+\t+\frac{d+1}2,\b)}(1-2 t^2) t^{2j} P_j^{(k+\a+\frac{d-2}2,\g)}\left(1- \frac{2s^2}{t^2}\right), & n = 2m+1.
\end{cases}
\end{align*}
Using the identity $P_n^{(\a,\b)}(1-2 s^2) = (-1)^n P_n^{(\b,\a)}(2 s^2-1)$, the two cases in the above can be combined, up to 
a constant multiple, in terms of the generalized Gegenbauer polynomial defined in \eqref{eq:gGegen}, which gives 
the stated basis. 
\end{proof}

By Corollary \ref{cor:FullSym}, the lemma leads to an orthogonal basis for $\CV_n( \Wb_{\a,\b,\g,\t}, \XX^{d+1})$ that consist of
$$
   C_{n-k-2j}^{(\b+ \f12, 2 j+k+\a+\g+\t+\f d2)}(t) t^{2j} P_j^{(\g, k+\a+\frac{d-2}2)}\left(\frac{2\|x\|^2}{t^2}-1\right)Y_\ell^k(x).
$$
Changing index $k \mapsto k-2j$, we sum up the result in the following proposition. 

\begin{prop} \label{prop:Va_a=1}
Let $\b,\g > -1$, $\a > -\f d2$, and $\a+\g+\t \ge -\f d2$. Let $\{Y_\ell^{k-2j}\}$ be an orthonormal basis of $\CH_n^d$. Then
the space $\CV_n(\Wb_{\a,\b,\g,\t}, \XX^{d+1})$ has an orthogonal basis consisting of 
\begin{align}\label{eq:basisCone}
  \Qb_{j,k,\ell}^n (x,t) \, & =  C_{n-k}^{(\b+ \f12, k+\a+\g+\t+\f d2)}(t) t^{2j} 
       P_j^{(\g, k-2j+\a+\frac{d-2}2)}\left(\frac{2\|x\|^2}{t^2}-1\right)Y_\ell^{k-2j}(x) \notag \\
    & =  \sP_j^{n-k+2j}\left(\sW^{(2k)}; \|x\|, t\right) Y_\ell^{k-2j}(x), \quad j \le k/2, \quad 0 \le k \le n 
 \end{align}
where $\sP_j^n(\sW^{(2k)})$ is given in Lemma \ref{lem:sWk_a=1}. Moreover, by \eqref{eq:OP_V_norm},
\begin{align} \label{eq:basisConeNorm}
  \Hb_{j,k}^n(\Wb_{\a,\b,\g,\t}):= \,& c_{\a,\b,\g,\t}  \int_{\XX^{d+1}} \left|\Qb_{j,k,\ell}^n (x,t)\right|^2 \Wb_{\a,\b,\g,\t} (x,t) \d x \d t \\
     = \, & \sh_{n-k}^{\b+ \f12, k+\a+\g+\t+\f d2)} h_{j}^{(\a + k-2j+\frac{d-2}2,\g)} 
     =: \sH_{j,n-k+2j}^n(\sW^{(2k)})
     \notag
\end{align}
in terms of the norm of $\sh_m^{(\a,\b)}$ in \eqref{eq:gGegenNorm} and $h_m^{(\a,\b)}$ in \eqref{eq:hn_ab}, where 
$c_{\a,\b,\g,\t}$ is the normalization constant of $\Wb_{\a,\b,\g,\t}$.
\end{prop}

In the case of $\a = 0$, this basis is studied in \cite{X21a} but written in terms of OPs on the unit ball, and 
further properties of the OPs are explored. These properties will be needed in the sequel and we describe them
in the following subsection.

\subsubsection{Gegenbauer polynomials on the double cone} \label{sec:Gegen-cone}
For $\a = 0$, the weight function in the previous section becomes 
$$
 \Wb_{\b,\g,\t}(x,t):= \Wb_{0,\b,\g, \t}(x,t) =  (1- t^2)^\b (t^2- \|x\|^2)^\g |t|^\t.
$$
The orthogonal basis for $\CV_n^{d+1}(\Wb_{\b,\g,\t},\XX^{d+1})$ in \eqref{eq:basisCone} consist of polynomials
\begin{align*}
\Qb_{j,k,\ell}^n (x,t) =  C_{n-k}^{(\b+ \f12, k+\g+\t+\f d2)}(t) t^{2j} 
     P_j^{(\g, k-2j+\frac{d-2}2)}\left(\frac{2\|x\|^2}{t^2}-1\right) Y_\ell^{k-2j}(x), 
\end{align*}
which can also be rewritten in terms of the classical OPs $P_{\ell,j}^k (W_{\g+\f12})$ in
\eqref{eq:basisBd} on the unit ball as
\begin{align*}
\Qb_{j,k,\ell}^n (x,t) = C_{n-k}^{(\b+ \f12, k+\g+\t+\f d2)}(t) t^k P_{\ell,j}^k \left(W_{\g+\f12};\frac{x}{t} \right).
\end{align*}
These are the polynomials given in \cite[(4.15)]{X21a}, under a change of parameters $(\b,\g,\mu) \mapsto (\t, \b+\f12, \g+\f12)$, 
where they are called the Gegenbauer polynomials on the double cone when $\t =\f12$. This basis is used to reveal two 
distinguished properties for $\CV_n^\sE(W_{\b,\g, \f12}, \XX^{d+1})$ and $\CV_n^\sO(W_{\b,\g, -\f12}, \XX^{d+1})$, which 
we describe below. Notice, however, that $\t= \f12$ for $\CV_n^\sE$ and $\t = -\f12$ for $\CV_n^\sO$. 

The first property is the spectral differential equation that has the space of OPs as eigenfunctions,
which is \cite[Theorem 4.6]{X21a}. 
 
\begin{thm}\label{thm:DEcone}
For $n=0,1,2,\ldots$, every $u \in \CV_n^\sE(\Wb_{\b,\g,\f12},\XX^{d+1})$ satisfies 
the differential equation 
\begin{align} \label{eq:DEcone}
&   \Big [(1-t^2) \partial_{t}^2 + \Delta_x - \la x,\nabla_x \ra^2 +\frac{2}{t} (1-t^2)\la x, \nabla_x\ra \partial_t 
    + (2\g+d+1)\frac{1}{t} \partial_t  \\
   &   \quad  -  t \partial_t - (2\b+2\g+d+1)\left( t \partial_t +  \la x ,\nabla_x\ra \right) \Big] u  
       = -n(n + 2 \b + 2 \g+d+1) u. \notag
\end{align}
Furthermore, every $u \in \CV_n^\sO(\Wb_{\b,\g,-\f12},\XX^{d+1})$ satisfies the differential equation 
\begin{align} \label{eq:DEConeO}
&   \Big [(1-t^2) \partial_{t}^2 + \Delta_x - \la x,\nabla_x\ra^2 - \la x,\nabla_x\ra 
       + \frac{2}{t} (1-t^2) \la x,\nabla_x \ra \partial_t \\
   &  \qquad\qquad +  \frac{2\g+d-1}{t} \Big(\partial_t -\frac{1}{t}\Big)  - 
      (2\g+2\g+d) \big(t \partial_t + \la x ,\nabla_x\ra \big) \Big] u   \notag   \\ 
   &  \qquad\qquad \qquad\qquad \qquad\qquad \qquad  = -n(n + 2 \b + 2 \g+d+1) u, \notag
\end{align}
where $\Delta_x$ and $\nabla_x$ indicate that the operators are acting on $x$ variable.
\end{thm}

The second property is an addition formula for the space of OPs, which is stated in \cite[Theorem 5.5]{X21a}. 

\begin{thm} \label{thm:PbEInt}
Let $\b, \g \ge 0$, $\b+\t \ge \f d 2$ and let $\a = \g+ \t+\frac{d}{2}$. For $n =1,2,\ldots$, 
\begin{align} \label{eq:POadd}
  \Pb_n^\sO \big (\Wb_{\b,\g,\t}; (x,t),(y,s) \big) = \frac{\a+\b+\f12}{\a+\f12} s t \, 
      \Pb_{n-1}^\sE \big (\Wb_{\b,\g,\t+1}; (x,t),(y,s) \big).
\end{align}
Furthermore, for  $\t = \f12$, $\b, \g \ge 0$ and $n =0,1,2,\ldots$, 
\begin{align} \label{eq:PEadd}
  \Pb_n^\sE  \big  (\Wb_{\b,\g,\f12}; (x,t),(y,s) \big)
     =    c_{\b} c_{\g} \, & \int_{[-1,1]^2}  Z_m^{\b+\g+\f{d+1}2} \big( \zeta(x,t,y,s; u,v)  \big) \\
          & \times (1-v^2)^{\b-\f12}(1-u^2)^{\g-\f12}  \d u \d v, \notag
\end{align}
where 
\begin{align*}
  \zeta(x,t, y,s; u, v) =  \left(\la x,y\ra  + u \sqrt{t^2-\|x\|^2}\sqrt{s^2-\|y\|^2} \right) \mathrm{sign}(st)
           + v \sqrt{1-s^2}\sqrt{1-t^2}.  
\end{align*}
The formula holds under the limit if either one of $\b$ and $\g$ is $-\f12$ or both are. 
\end{thm}

These two properties will be used to derive analogs for several new domains in the next section. For $\t > \f12$, 
there is also an addition formula for $\Pb_n^\sE \big (W_{\b,\g,\t}; (x,t),(y,s) \big)$, more involved and given as 
an integral over $[-1,1]^4$ in \cite[Theorem 5.6]{X21a}, which we shall not state to avoid an overload of formulas. 

As we discussed in the Subsection \ref{set:4.2}, these two properties can be used to study the Fourier orthogonal series 
for functions even in the $t$-variable on $\XX^{d+1}$ or functions on $\XX_+^{d+1}$. For some of their applications 
in approximation theorem and computational harmonic analysis, see \cite{X21a, X23a}. 

\subsection{Double conic domains}\label{sect:5.1}
For $\fa > 0$, we consider the domain $\sqrt{\Omega_\fa}$ defined as a triangle with vertices at $(0,0)$, $(0, \fa)$ and $(1,1)$. 
More precisely, 
$$
 \nabla_\fa := \sqrt{\Omega_\fa} = \left \{(u,v):  0 \le  u \le v  \le (1-\fa) u + \fa \le 1 \right\}, \qquad   \fa \ge 0.  
$$
The rotation in the $t$ axis of the fully symmetric domain $\Omega_\fa =\left\{(s,t): (s^2,t^2) \in \sqrt{\Omega_\fa}\right\}$ 
leads to the conic domains of two bodies, 
$$
   \XX_{\fa}^{d+1} =  \left\{(x,t): \|x\|^2 \le t^2 \le (1-\fa) \|x\|^2 + \fa, \,  |t| \le 1\right\}.  
$$
If $\fa = 1$, this reduces to the double cone studied in Subsections \ref{sect:double_cone} and \ref{sec:Gegen-cone}.

For $0 < \fa <1$, the domain $\XX_{\fa=1}^{d+1}$ is bounded by the double conic surface and is caped by hyperbolic 
surfaces $\{(x,t): (1-\fa) \|x\|^2 + \fa = t^2 \le 1\}$ at the two ends. For $\fa = \f12$, the domain $\Omega_\fa$ 
is depicted in the left-hand of Figure \ref{fig:double_cone}, where $\sqrt{\Omega_\fa}$ is the shaded triangle, and 
the domain $\XX_\fa^{d+1}$ is depicted in the right-hand side. 
\begin{figure}[htb]
\hfill
\begin{minipage}[b]{0.44\textwidth} \centering
\includegraphics[width=0.98\textwidth]{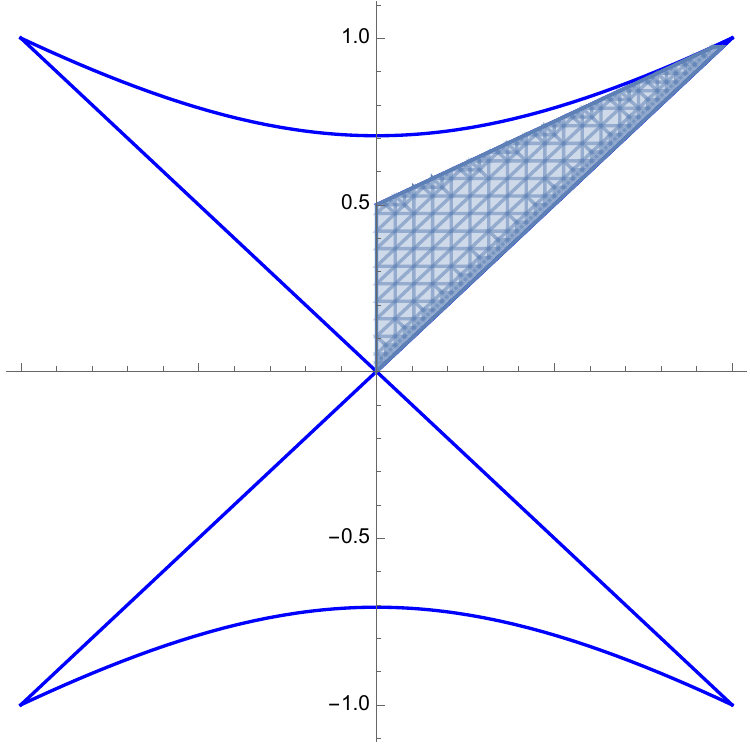}
\end{minipage}\hfill
\begin{minipage}[b]{0.56\textwidth}
\centering
\includegraphics[width=0.64\textwidth]{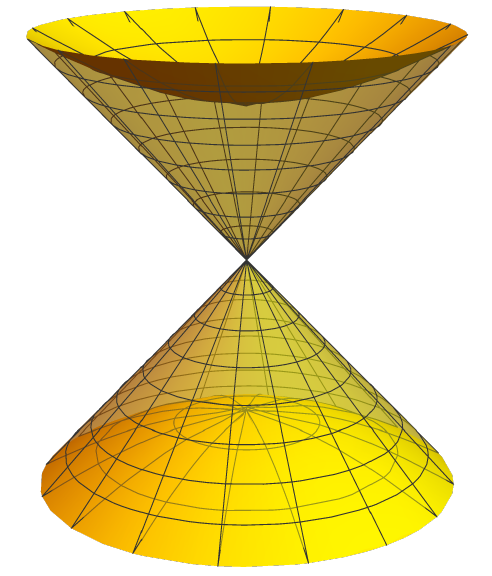}
\end{minipage}
\hspace*{\fill}
\caption{Double conic domain with $\fa = \f12$} \label{fig:double_cone}
\end{figure}

For $\fa > 1$, the domain $\XX_{\fa=1}^{d+1}$ is again bounded by the double conic surface but is capped at the two 
ends by the surfaces $\{(x,t): (1-\fa) \|x\|^2 + \fa = t^2 \le 1\}$ that are elliptic rather than hyperbolic .
For $\fa = 2$, the domains $\sqrt{\Omega_\fa}$ and $\XX_\fa^{d+1}$ are depicted in the Figure \ref{fig:double_cone_ell}. 
\begin{figure}[htb]
\hfill
\begin{minipage}[b]{0.44\textwidth} \centering
\includegraphics[width=0.98\textwidth]{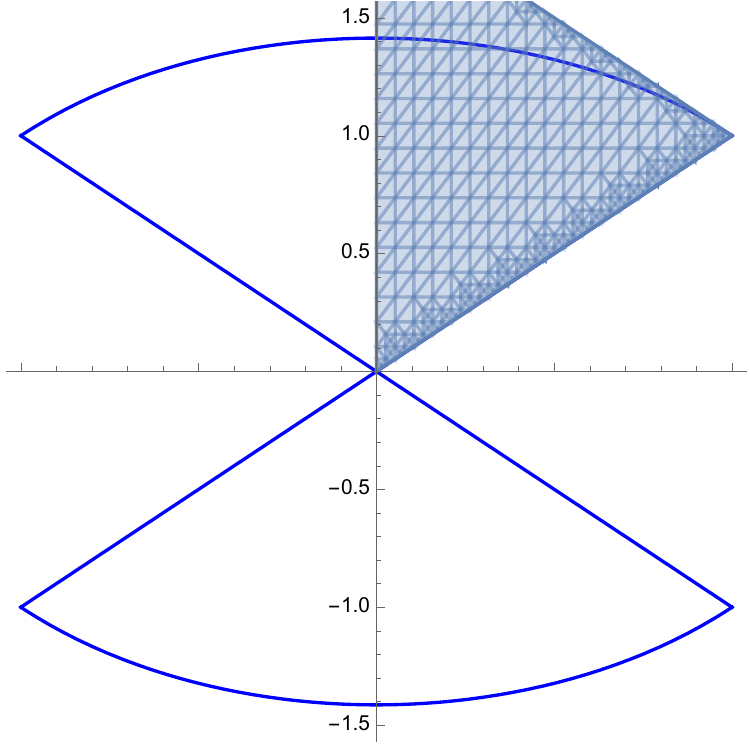}
\end{minipage}\hfill
\begin{minipage}[b]{0.56\textwidth}
\centering
\includegraphics[width=0.76\textwidth]{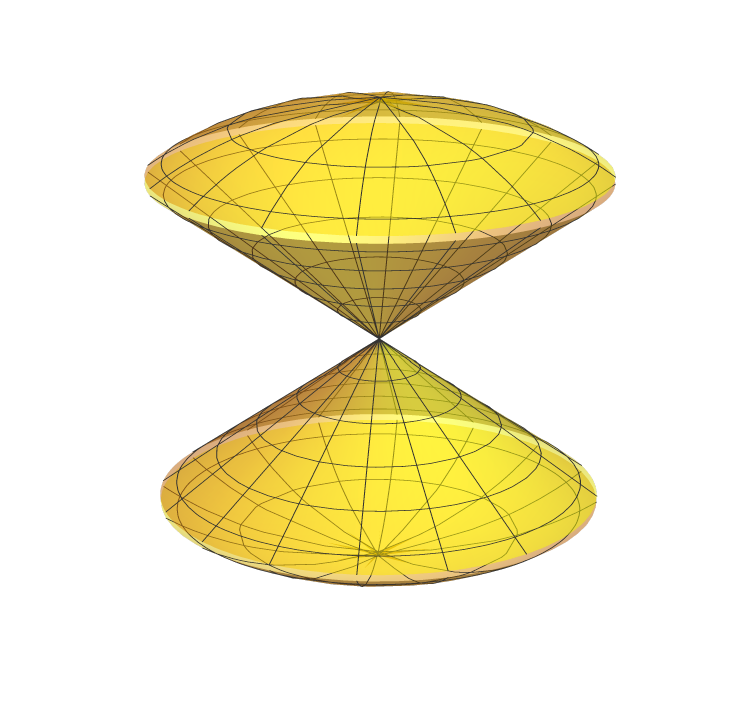}
\end{minipage}
\hspace*{\fill}
\caption{Double conic domain with $\fa = \f12$} \label{fig:double_cone_ell}
\end{figure}

For $\a, \b,\g, \t  >  -1$, let $\sw^\fa_{\a,\b,\g,\t}$ be the weight function defined by
$$
  \sw^\fa_{\a,\b,\g,\t}(s,t) = s^\a \big(\fa (1-s)- (t-s)\big)^\b (t-s)^\g \big(t - (1-\fa)s\big)^{\t}, \quad (s,t) \in \sqrt{\Omega_\fa},
$$
so that the weight function on the conic domain $\XX^{d+1}$ becomes
\begin{align}\label{eq:W-3paraHyp}
&  \Wb_{\a,\b,\g,\t}^\fa (x,t) \ = \sw^\fa_{\a,\b,\g,\t}\left(\|x\|^2,t^2\right)\\
    &\qquad = \|x\|^{2\a} \left(\fa + (1-\fa)\|x\|^2 - t^2\right)^\b\left(t^2- \|x\|^2\right)^\g \big (t^2-(1-\fa)\|x\|^2 \big)^{\t}. \notag
\end{align}
In this case, the triangle $\nabla_\fa$ becomes the triangle $\nabla$ in \eqref{eq:tri-nabla} under the affine mapping
$(s,t) \mapsto \big(s, s+\frac{t-s}{\fa} \big)$ and the weight function $\sw^\fa$ becomes
$$
   \sw^\fa_{\a,\b,\g,\t}(s,t) =  \fa^{\b+\g+\t} \sw_{\a,\b,\g,\t} \left(s, s+\tfrac1{\fa}(t-s)\right)
$$
in terms of $\sw_{\a,\b,\g,\t}$ defined in \eqref{eq:w-4parameters}. Since the affine change of variable is applied to 
the weight function $\sW_{- \f12, \pm \f12}^{(k)}(s,t)$, we could deduce the basis for $\CV_m(\sW_\fa^{(2k)}, \nabla_\fa)$
from the corresponding basis for $\fa = 1$ by a changing variable. Indeed, by \eqref{eq:sW-cone}, the 
weight function $\sW_{- \f12, \pm \f12}^{(k)}$ for $\fa \ne 0$ becomes 
$$
\sW_{k+\a+\frac{d-2}2, \b, \g, \t\pm \f12} \left(s, 1-s-\tfrac1{\fa}(t-s)\right).
$$
Hence, following the same argument used for $\fa = 1$, we obtain the following analog of Proposition \ref{prop:Va_a=1A}. 

\begin{prop} \label{prop:Va_a!=1A}
Let $\b,\g > -1$, $\a > -\f d2$, and $\a+\g+\t \ge -\f d2$. Let $\{Y_\ell^{k-2j}\}$ be an orthonormal basis of $\CH_n^d$. Then
the space $\CV_n(\Wb_{\a,\b,\g,\t}^\fa, \XX_\fa^{d+1})$ has an orthogonal basis that consists of 
\begin{equation*}
 Y_\ell^k (x) \begin{cases}
  \sT_{j,m}^{k+\a+\frac{d-2}2, \b, \g, \t -\f12} \left(\|x\|^2, 1-\|x\|^2-\frac1 \fa (t^2- \|x\|^2) \right), & n-k = 2m, \\
 t \sT_{j,m}^{k+\a+\frac{d-2}2, \b, \g, \t + \f12}\left(\|x\|^2, 1-\|x\|^2-\frac1 \fa (t^2- \|x\|^2) \right), & n-k = 2m+1 
\end{cases}
\end{equation*}
for $1 \le \ell \le \dim \CH_{k-2j}^d$ and $0 \le j \le m$.
\end{prop}

There is however a subtlety for an analog of Proposition \ref{prop:Va_a=1}. Let us start with an analog of 
Lemma \ref{lem:sWk_a=1} for $\sW_\fa^{(2k)}$ associated with $\Wb^\fa_{\a,\b,\g,\t}$, which holds however 
only for polynomials of even degrees. 

\begin{lem}\label{lem:sWk_a!=1}
Let $\sP_j^{n}(\sW^{(2k)})$ be the orthogonal basis \eqref{eq:basisCone} for $\CV_{n}^\sE(\sW^{(2k)}, \Omega)$.
Then an orthogonal basis for $\CV_n^\sE(\sW_\fa^{(2k)}, \Omega_\fa)$ consists of polynomials
\begin{align*} 
  \sP_j^n\left(\sW_\fa^{(2k)}; s,t\right) & = C_{n-2j}^{(\b+ \f12, 2 j+k+\a+\g+\t+\f {d} 2)}\left(\sqrt{s^2+\tfrac1{\fa}\big(t^2- s^2\big)}\right) 
        \left(s^2+\tfrac1{\fa}\big(t^2- s^2\big)\right)^j \\
         &\qquad\qquad \times P_j^{(\g, k+\a+\frac{d-2}2)}\left(\frac{2s^2}{s^2+\tfrac1{\fa}\big(t^2- s^2\big)}-1\right) \\
        &  = \sP_j^{n} \left(W^{(2k)}; s, \sqrt{s^2+\tfrac1{\fa}\big(t^2- s^2\big)}\right), \quad 0 \le j \le n, \quad 0 \le j \le n/2,
\end{align*}
only if $n$ is an even integer. Moreover, the norm of $\sP_j^n\big(\sW_\fa^{(2k)}\big)$  in 
$L^2\big(\sW_\fa^{(2k)}, \Omega_\fa\big)$ and the norm of $\sP_j^{n}\big(\sW^{(2k)}\big)$ in 
$L^2\big(\sW_\fa^{(2k)}, \Omega_\fa\big)$ are equal.
\end{lem}

\begin{proof}
The first identity of the OP is verified exactly as in the proof of Lemma \ref{lem:sWk_a=1}. If $n$ is even, then
$C_{n-2j}^{(\l,\mu)}$ is an even function, so that $C_{n-2j}^{(\l,\mu)}(\sqrt{x})$ is a polynomial in $x$. Hence, 
$\sP_j^{n} \big(W_\fa^{(2k)}\big)$ is indeed a polynomial in $s$ and $t$. For $n$ is odd, however,
$\sP_j^{n} \big(W_\fa^{(2k)}\big)$ is no longer a polynomial as $C_{n-2j}^{(\l,\mu)}$ is odd and $C_{n-2j}^{(\l,\mu)}(\sqrt{x})$ 
contains a factor $\sqrt{x}$. To see that their norm squares are equal, we observe that the norm
of $\sP_j^{n}\big(W^{(2k)}\big)$ can be reduced to the norm square of $P_{j,m}(\sW_{-\f12, \pm \f12})$, as in the proof of 
Theorem \ref{thm:FullSym}. Since $\sW_{-\f12, \pm \f12}^\fa$ becomes $\sW_{-\f12, \pm \f12}$ by an affine change of 
variable, the norm $P_{j,m}(\sW_{-\f12, \pm \f12}^\fa)$ is the same as that of $P_{j,m}(\sW_{-\f12, \pm \f12})$. 
\end{proof}

Using the basis in Lemma \ref{lem:sWk_a!=1}, we can deduce by \eqref{eq:basisE} and \eqref{eq:OP_fullsym} an 
orthogonal basis for the space $\CV_n^\sE\left(\Wb_{\a,\b,\g,\t}^\fa, \XX_\fa^{d+1}\right)$ of OPs 
that are even in $t$ variable.

\begin{prop}\label{prop:OPDouble_Va}
The space $\CV_n^\sE\left(\Wb_{\a,\b,\g,\t}^\fa, \XX_\fa^{d+1}\right)$ has an orthogonal basis 
\begin{align}\label{eq:OPDouble_Va}
   \Qb_{j,n-2m,\ell}^{\fa,n}(x,t)\, & =  \sP_j^{2m} \left(\sW^{(2k)}; \|x\|, \sqrt{\|x\|^2+\tfrac{1}{\fa} \big(t^2-\|x\|^2\big)} \right) Y_\ell^{k-2j}(x) \\ 
       & =  \Qb_{j, n-2m,\ell}^{n} \left(x, \sqrt{\|x\|^2+\tfrac{1}{\fa} \big(t^2-\|x\|^2\big)} \right), \notag
\end{align}
where $\Qb_{j,k,\ell}^{n}$ are polynomials given in \eqref{eq:basisCone}. Moreover, the norm of $ \Qb_{j,n-2m,\ell}^{\fa,n}$
in $L^2\left(\Wb_{\a,\b,\g,\t}^\fa, \XX_\fa^{d+1}\right)$ is equal to the norm of $\Qb_{j,n-2m,\ell}$ in 
$L^2\left(\Wb_{\a,\b,\g,\t}, \XX^{d+1}\right)$.
\end{prop}

For $\fa \ne 1$, an orthogonal basis for the space $\CV_n^\sO\left(\Wb_{\a,\b,\g}^\fa, \XX^{d+1}\right)$ can be 
derived from Proposition \ref{prop:Va_a!=1A}, but the analog of \eqref{eq:OPDouble_Va} no longer holds. The advantage
of the basis in \eqref{eq:OPDouble_Va} lies in the spectral operator and the addition formula that we now state. The 
first one is the existence of the spectral differential operator. 

\begin{thm} \label{thm:DEcone_a!=1}
For $\fa > 0$ and $n=0,1,2,\ldots$, there is a second order differential operator $\CD^{\fa}_{x,t}$ such that 
every $u \in \CV_n^\sE(\Wb^\fa_{\b,\g,\f12},\XX^{d+1})$ satisfies the differential equation 
\begin{align*}
 \CD^{\fa}_{x,t} u = -n(n + 2 \b + 2 \g+d+1) u,
\end{align*}
where if $\CD_{x,t}(\partial_{x}, \partial_t)$ denotes the operator in the left-hand side of \eqref{eq:DEcone}, 
then $\CD^{\fa}_{x,t} = \CD^\fa_{x,t}(\partial_x, \partial_t)$ is determined by
$$
   \CD^\fa_{x,t}(\partial_x, \partial_t) = \CD_{x, z}(\partial_{x}, \partial_z), \qquad z = \sqrt{\|x\|^2+\tfrac{1}{\fa} \big(t^2-\|x\|^2\big)}.
$$
\end{thm}

This follows as a consequence of the second identity in \eqref{eq:OPDouble_Va}, so is the next property on the addition
formula. 

\begin{thm} \label{thm:PbEInt_a}
For  $\fa > 0$, $\t = \f12$, $\b, \g \ge 0$, the kernel $\Pb_n^\sE  \big  (\Wb^\fa_{\b,\g,\f12}; \cdot, \cdot\big)$ satisfies
the addition formula \eqref{eq:PEadd} with $\zeta(x,t,y,s;u,v)$ replaced by 
\begin{align*}
  \zeta^\fa (x,t, y,s; u, v) = \,& \left(\la x,y\ra  + \frac{u }{\fa} \sqrt{t^2-\|x\|^2}\sqrt{s^2-\|y\|^2} \right) \mathrm{sign}(st) \\
    & + v \sqrt{1-\|y\|^2-\tfrac1{\fa}\big(s^2- \|y\|^2\big)}\sqrt{1-\|x\|^2-\tfrac1{\fa}\big(t^2- \|x\|^2\big)}.  
\end{align*}
The formula holds under the limit if either one of $\b$ and $\g$ is $0$ or both are. 
\end{thm}

\subsection{Double hyperbolic domains} 
As a further generation, we consider the triangle that has vertexes at $(0,\fa), (0,\fb)$ and $(1,1)$ with $\fa > \fb > 0$; that is, 
$$
 \nabla_{\fa,\fb} := \sqrt{\Omega_{\fa,\fb}} = \left \{(u,v): 0 \le (1-\fb) u+\fb \le v \le (1-\fa) u + \fa  \le 1 \right\},
$$
which reduces to the triangle in the previous subsection when $\fb = 0$. The rotation in the $t$ axis of the 
fully symmetric $\Omega_{\fa,\fb} =\{(s,t): (s^2,t^2) \in \sqrt{\Omega_{\fa,\fb})}\}$ leads to the double hyperboloid
$$
   \XX_{\fa,\fb}^{d+1} =  \left\{(x,t): (1-\fb) \|x\|^2 + \fb \le t^2 \le (1-\fa) \|x\|^2 + \fa  \le 1 \right\}.  
$$

\subsubsection{$\fa=1$} The domain $\XX_{\fa,1}^{d+1}$ is the hyperboloid bounded by the hyperbolic surface 
$\{(x,t): (1-\fb) \|x\|^2 + \fb = t^2\}$ and two hyperplanes at $t= \pm 1$, 
$$
   \XX_{1,\fb}^{d+1} =  \left\{(x,t): (1-\fb) \|x\|^2 + \fb \le t^2 \le 1 \right\}.  
$$
For $\fb = 0.1$, this domain and $\Omega_{1,\fb}$ are depicted in Figure \ref{fig:double_hypo}, where 
$\nabla_{1,\fb}$ is the shaded triangle. 
\begin{figure}[htb]
\hfill
\begin{minipage}[b]{0.44\textwidth} \centering
\includegraphics[width=1\textwidth]{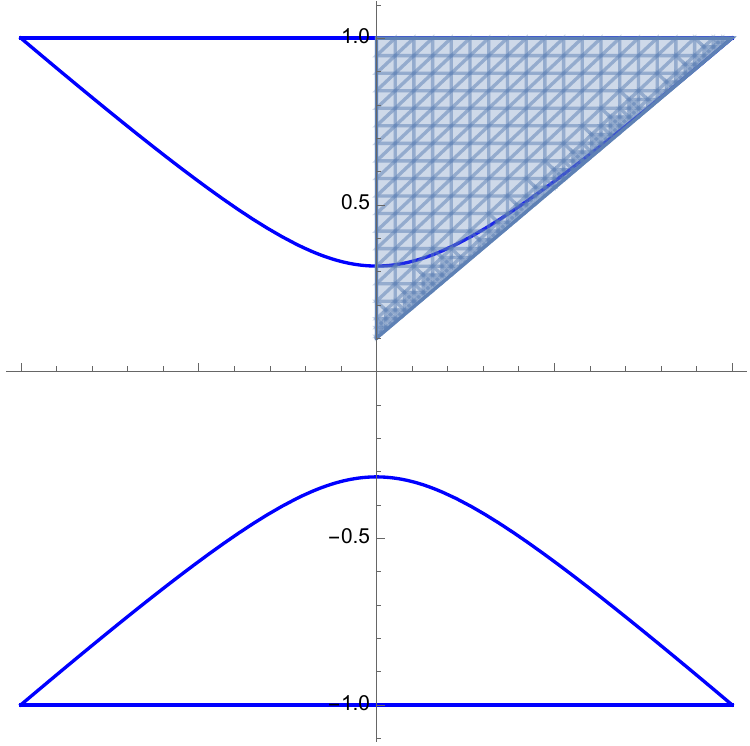}
\end{minipage}\hfill
\begin{minipage}[b]{0.56\textwidth}
\centering
\includegraphics[width=0.62\textwidth]{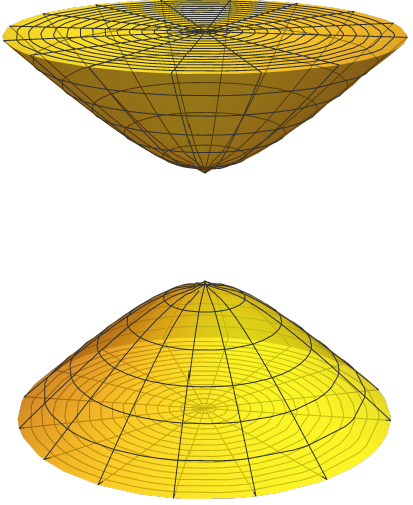}
\end{minipage}
\hspace*{\fill}
\caption{Double hyperboloid with $\fb= 0.1$} \label{fig:double_hypo}
\end{figure}
OPs for a family of weight functions on the hyperboloid $\XX_{\fa,1}^{d+1}$ were studied 
in \cite{X21a} but with a different parametrization that amounts to a dilation of $t \mapsto \sqrt{1+\rho}\, t$ with 
$\fb = \rho /(1+\rho)$, where it is shown that orthogonal bases can be derived from the corresponding basis
on the double cone $\XX_{1,1}^{d+1}$ by a simple change of variables, which we shall also discuss below. 

\subsubsection{$\fa \ne 1$} \label{sec:5.2.2}
There are three cases according to different geometry of the domain.

{\it Case 1. $0< \fb < \fa < 1$}. The domain $\XX_{\fa,\fb}^{d+1}$ is a hyperboloid bounded by the hyperbolic surface 
$\{(x,t): (1-\fb) \|x\|^2 + \fb = t^2 \le 1\}$ and capped at both ends by the hyperbolic surface $\{(x,t): (1-\fa) \|x\|^2 + \fa = t^2 \le 1\}$
instead of hyperplanes. For $\fa = \f58$ and $\fb = \f18$, the domains $\Omega$ and $\XX_{\fa,\fb}^{d+1}$ are depicted in
Figure \ref{fig:double_hypo_hypo}. 
\begin{figure}[htb]
\hfill
\begin{minipage}[b]{0.44\textwidth} \centering
\includegraphics[width=1\textwidth]{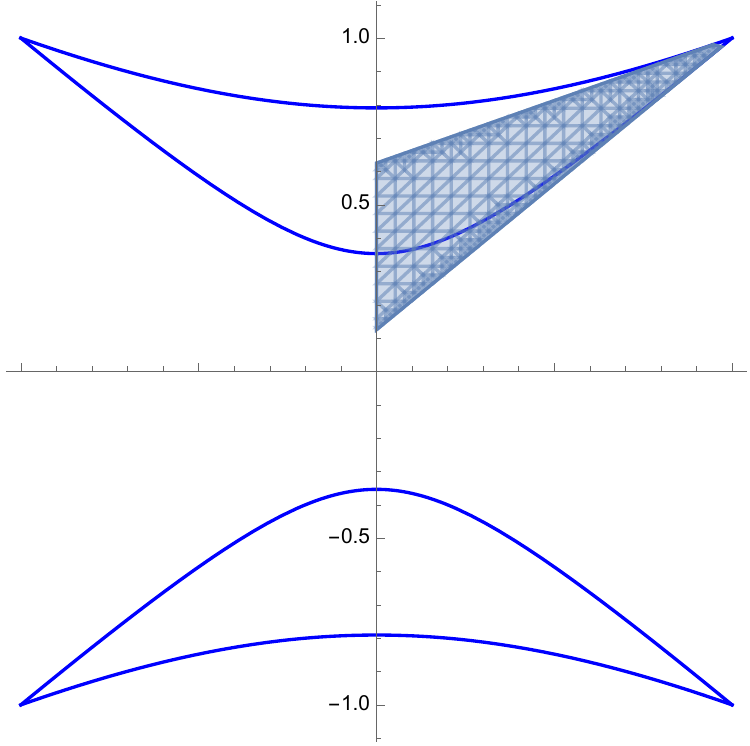}
\end{minipage}\hfill
\begin{minipage}[b]{0.56\textwidth}
\centering
\includegraphics[width=0.62\textwidth]{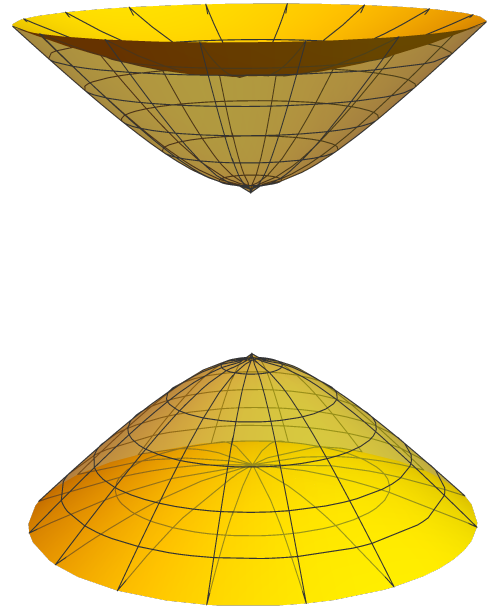}
\end{minipage}
\hspace*{\fill}
\caption{Double hyperboloid with $\fa= \f58$ and $\fb=\f18$} \label{fig:double_hypo_hypo}
\end{figure}

{\it Case 2. $0 < \fb < 1 <  \fa$}.  The domain $\XX_{\fa,\fb}^{d+1}$ is the hyperboloid bounded by the same hyperbolic 
surface $\{(x,t): (1-\fb) \|x\|^2 + \fb = t^2 \le 1\}$ but capped at both ends by the surface $\{(x,t): (1-\fa) \|x\|^2 + \fa = t^2 \le \fa\}$
that is elliptic rather than hyperbolic. For $\fa = 2$ and $\fb = \f18$, the domains $\Omega$ and $\XX_{\fa,\fb}^{d+1}$ 
are depicted in Figure \ref{fig:double_hypo_ell}.
\begin{figure}[htb]
\hfill
\begin{minipage}[b]{0.44\textwidth} \centering
\includegraphics[width=0.95\textwidth]{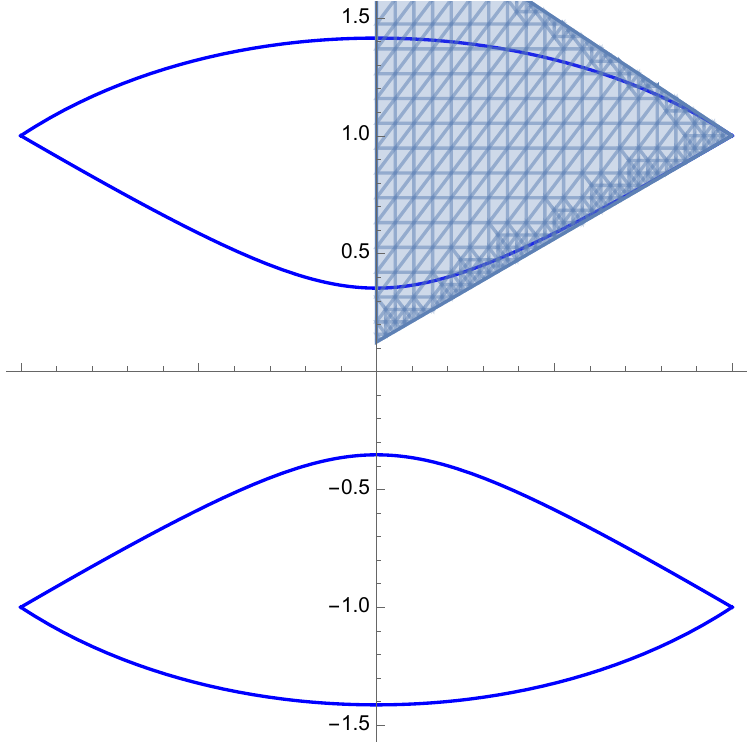}
\end{minipage}\hfill
\begin{minipage}[b]{0.56\textwidth}
\centering
\includegraphics[width=0.62\textwidth]{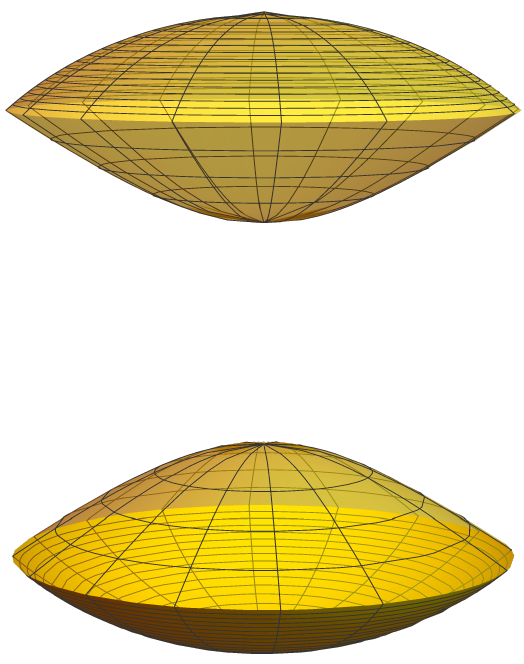}
\end{minipage}
\hspace*{\fill}
\caption{Double hyperboloid with $\fa= 2$ and $\fb=\f18$} \label{fig:double_hypo_ell}
\end{figure}

{\it Case 3. $0 < 1 < \fb < \fa$}. The domain $\XX_{\fa,\fb}^{d+1}$ is bounded by two double surfaces 
$\{(x,t): (1-\fb) \|x\|^2 + \fb = t^2 \le \fa\}$ and $\{(x,t): (1-\fa) \|x\|^2 + \fa = t^2 <1\}$,  both are elliptic rather than 
hyperbolic. The two surfaces intersect at the unit sphere $\sph$. Geometrically, both surfaces that bound the 
upper half of $\XX_{\fa,\fb}$ concave downward. We shall not depict the region in this case. 

\medskip

Let $\sw_{\a,\b,\g,\t}$ be the weight function defined in \eqref{eq:w-4parameters}. For $(s,t) \in \nabla_{\fa,\fb}$, let
\begin{align*}
  \sw^{\fa,\fb}_{\a,\b,\g,\t}(s,t) \, & = s^{\a} \big(\fa (1-s) - (t -s)\big)^\b (t- s - \fb (1-s))^\g \big(t-\fb - (1-\fa)s\big)^{\t} \\
        & = (\fa-\fb)^{\b+\g+\t} \sw_{\a,\b,\g,\t}\left(s, \frac{t-\fb- (1-\fa)s}{\fa-\fb}\right).
\end{align*}
Then the corresponding weight function on the rotational solid $\XX_{\fa,\fb}^{d+1}$ is given by
\begin{align*}
 \Wb_{\a,\b,\g,\t}^{\fa,\fb}(x,t) \, & = \|x\|^{2 \a} \big(\fa(1-\|x\|^2)- (t^2 - \|x\|^2) \big)^\b  \\
        & \times     \big(t^2- \|x\|^2 - \fb (1-\|x\|^2)\big)^\g \big(t^2-\fb - (1-\fa)\|x\|^2\big)^{\t}.
 \end{align*}
In terms of the $\Wb_{\a,\b,\g,\t}$ defined in \eqref{eq:W-4paraHyp}, we have
\begin{align*}
\Wb_{\a,\b,\g,\t}^{\fa,\fb}(x,t) = (\fa-\fb)^{\b+\g+\t} \Wb_{\a,\b,\g,\t} \left(\|x\|^2, \frac{t^2-\fb-(1-\fa) \|x\|^2}{\fa-\fb} \right).
\end{align*}

If $\fa = 1$, then an orthogonal basis for $\CV_n(\Wb_{\a,\b,\g,\t}^{1,\fb}, \XX_{1,\fb}^{d+1})$ follows immediately from 
the corresponding basis on the double cone $\XX^{d+1}$ in Proposition \ref{prop:Va_a=1A}. More precisely, 
from \eqref{prop:Va_a=1}, we obtain the following proposition. 

\begin{prop}
Let $\b,\g > -1$, $\a > -\f d2$, and $\a+\g+\t \ge -\f d2$. Let $\{\Qb_{j,k,\ell}^n \}$ be the orthogonal basis of 
the space $\CV_n(\Wb_{\a,\b,\g,\t}, \XX^{d+1})$ in Proposition \ref{prop:Va_a=1}. Then the polynomials 
$$
    \Qb_{j,k,\ell}^{(1,\fb), n} (x,t) =  \Qb_{j,k,\ell}^n \left(x, \sqrt{\frac{t^2-\fb}{1-\fb} }\right), 
    \quad 1 \le \ell \le \dim \CH_{k-2j}^d, \,\, j \le k/2, \,\, 0 \le k \le n,
$$
consist of an orthogonal basis for $\CV_n\big(\Wb^{1,\fb}_{\a,\b,\g,\t}, \XX^{d+1}_{1,\fb}\big)$. 
\end{prop}

In the case $\a = 0$, $\g = \mu-\f12$ and $\t = \f12$, this basis is studied in \cite{X21a} with $\fb = \rho^2/(1+\rho^2)$ and 
$t \mapsto \sqrt{1+\rho^2}\, t$. In particular, for $\a = 0$, the full strength of Theorem \ref{thm:DEcone} for the spectral
differential operator and Theorem \ref{thm:PbEInt} for the addition formula hold, which we summarize in the 
following proposition for the record. 

\begin{thm} \label{thm:DEcone_a!=1b}
Let $\fa = 1$ and $0 \le \fb < 1$. Let $\CD_{x,t}(\partial_{x}, \partial_t)$ denote the operator in the left-hand side 
of \eqref{eq:DEcone} or \eqref{eq:DEConeO}, respectively. Define 
$$
   \CD^{1,\fb}_{x,t}(\partial_x, \partial_t) = \CD_{x, z}(\partial_{x}, \partial_z), \qquad z = \sqrt{\frac{t^2-\fb}{1-\fb}}. 
$$
Then every $u \in \CV_n^\sE(\Wb^{1,\fb}_{0,\b,\g,\f12}, \XX_{1,\fb}^{d+1})$ or 
$u \in \CV_n^\sO(\Wb^{1,\fb}_{0,\b,\g,-\f12}, \XX_{1,\fb}^{d+1})$, respectively, satisfies  
\begin{equation} \label{eq:CD1b}
 \CD^{1,\fb}_{x,t}(\partial_x, \partial_t) u = - n (n+2\b+\g+ 1) u .
\end{equation}
\end{thm}

\begin{thm} \label{thm:PbEInt_ab}
Let $\fa =1$ and $0 < \fb < 1$. Then the kernel $\Pb_n^\sE  \big  (\Wb^{1,\fb}_{0,\b,\g,\t}; \cdot, \cdot\big)$ and the
kernel $\Pb_n^\sO \big  (\Wb^{1,\fb}_{0,\b,\g,\t+1}; \cdot, \cdot\big)$ satisfies the relation \eqref{eq:POadd}. Moreover,
$\Pb_n^\sE  \big  (\Wb^{1,\fb}_{0,\b,\g,\t}; \cdot, \cdot\big)$ satisfies the addition formula \eqref{eq:PEadd}
with $\zeta(x,t,y,s;u,v)$ replaced by 
\begin{align*}
  \zeta^{1,\fb} (x,t, y,s; u, v) = \,&
      \left(\la x,y\ra  + u \sqrt{ \tfrac{t^2-\fb}{1-\fb}-\|x\|^2}\sqrt{\tfrac{s^2-\fb}{1-\fb}-\|y\|^2} \right) \mathrm{sign}(st) \\
    & + \frac{v}{1-\fb} \sqrt{1- s^2}\sqrt{1-t^2}.  
\end{align*}
\end{thm}
  
For $\fa \ne 1$ and $\fb > 0$, orthogonal bases on the domain $\XX^{d+1}_{\fa,\fb}$ can be derived from the 
orthogonal bases on the domain $\XX^{d+1}_{1, \fb}$, just like in the case in Section \ref{sect:5.1} for $\fb =0$. 
In this case, we only have an explicit orthogonal basis for the space of OPs that are even
in the $t$-variable. 

\begin{prop}
Let $\b,\g > -1$, $\a > -\f d2$, and $\a+\g+\t \ge -\f d2$. Let $\{\Qb_{j,k,\ell}^n \}$ be the orthogonal basis of 
the space $\CV_n^\sE(\Wb_{\a,\b,\g,\t}, \XX^{d+1})$ in Proposition \ref{prop:OPDouble_Va}. Then the polynomials 
$$
    \Qb_{j,k,\ell}^{(\fa,\fb), n} (x,t) =  \Qb_{j,k,\ell}^n \left(x, \sqrt{\frac{t^2-\fb- (1-\fa)\|x\|^2}{\fa-\fb}}\right), 
     \quad j \le k/2, \,\, 0 \le k \le n
$$
and $1 \le \ell \le \dim \CH_{k-2j}^d$, consist of an orthogonal basis for $\CV_n^\sE\big(\Wb^{\fa,\fb}_{\a,\b,\g,\t}, \XX^{d+1}_{\fa,\fb}\big)$. 
\end{prop}

This provides an orthogonal basis for $\CV_n^\sE \big(\Wb^{\fa,\fb}_{\a,\b,\g,\t}, \XX_{\fa,\fb}^{d+1}\big)$ given 
explicitly in terms of classical Jacobi polynomials on the simplex and spherical harmonics. Such a basis, however,
does not hold for $\CV_n^\sO\big(\Wb^{\fa,\fb}_{\a,\b,\g,\t}, \XX_{\fa,\fb}^{d+1}\big)$ for $\fa \ne 1$.

We also have analogs of the spectral differential operators and the addition formula in this setting, both hold for
polynomials in $\CV_n^\sE \big(\Wb^{\fa,\fb}_{\a,\b,\g,\t}, \XX_{\fa,\fb}^{d+1}\big)$ with $\t = \f12$. 

\begin{thm}
Let $0 < \fb < \fa$. Then every $u \in \CV_n^\sE \big(\Wb^{\fa,\fb}_{0,\b,\g,\f12}, \XX_{\fa,\fb}^{d+1}\big)$ satisfies
an analog of \eqref{eq:CD1b} with $\CD_{x,t}^{1,\fb}$ replaced by 
$$
 \CD^{\fa,\fb}_{x,t}(\partial_x, \partial_t) = \CD_{x, z}(\partial_{x}, \partial_z), \qquad z = \sqrt{\frac{t^2-\fb- (1-\fa)\|x\|^2}{\fa-\fb}}. 
$$
Moreover, the addition formula \eqref{eq:PEadd} holds for the kernel $\Pb_n^\sE  \big  (\Wb^{\fa,\fb}_{0,\b,\g,\f12}; (x,t),(y,s) \big)$ 
with $\zeta(x,t,y,s;u,v)$ replaced by 
\begin{align*}
  \zeta_{\fa,\fb} (x,t, y,s; u, v) = \,&
      \Big(\la x,y\ra  + \frac{u}{\fa-\fb} \sqrt{t^2-\fb- (1-\fb)\|x\|^2}\sqrt{s^2-\fb- (1-\fb)\|y\|^2} \Big) \mathrm{sign}(st) \\
    & + \frac{v}{\fa - \fb} \sqrt{\fa - t^2+(1-\fa)\|x\|^2}\sqrt{\fa-s^2+ (1-\fa)\|y\|^2}.  
\end{align*}
\end{thm}

\subsection{Intersections of two touching ellipsoids} 
For $0 \le \fb < \fa$, we consider the domain $\sqrt{\Omega_{\fa,\fb}}$ defined as a triangle with vertices at 
$(0,\fa)$, $(0, \fb)$ and $(1,0)$. More precisely,
$$
   \sqrt{\Omega} = \left \{(u,v):  \fb (1-u) \le v \le \fa(1 - u) \right\}. 
$$
The rotation in the $t$ axis of the fully symmetric domain $\Omega =\{(s,t): (s^2,t^2) \in \sqrt{\Omega}\}$ 
leads to the domain bounded by two ellipsoidal surfaces that touch at the unit sphere $\sph$ when $t =0$,
$$
      \XX^{d+1} = \left\{(x,t):  \sqrt{1- \frac{t^2}{\fb}}\le  \|x\| \le  \sqrt{1- \frac{t^2}{\fa}}, \, \, |t| \le \fa \right\}. 
$$
The ellipsoidal surface $\{(x,t):  \fa(1- \|x\|^2) = t^2\}$ becomes the unit sphere $\SS^d$ if $\fa = 1$. The domain 
$\Omega_{\fa,\fb}$ and $\XX_{\fa,\fb}^{d+1}$ are depicted on the right-hand side of Figure \ref{fig:ellipsoid}
for $\fa = 1$ and $\fb = \f12$, where $\Omega_{\fa,\fb}$ is the shaded domain between  the circle and 
the ellipse, and $\sqrt{\Omega_{\fa,\fb}}$ is the shaded triangle domain.   
\begin{figure}[htb]
\hfill
\begin{minipage}[b]{0.44\textwidth} \centering
\includegraphics[width=0.92\textwidth]{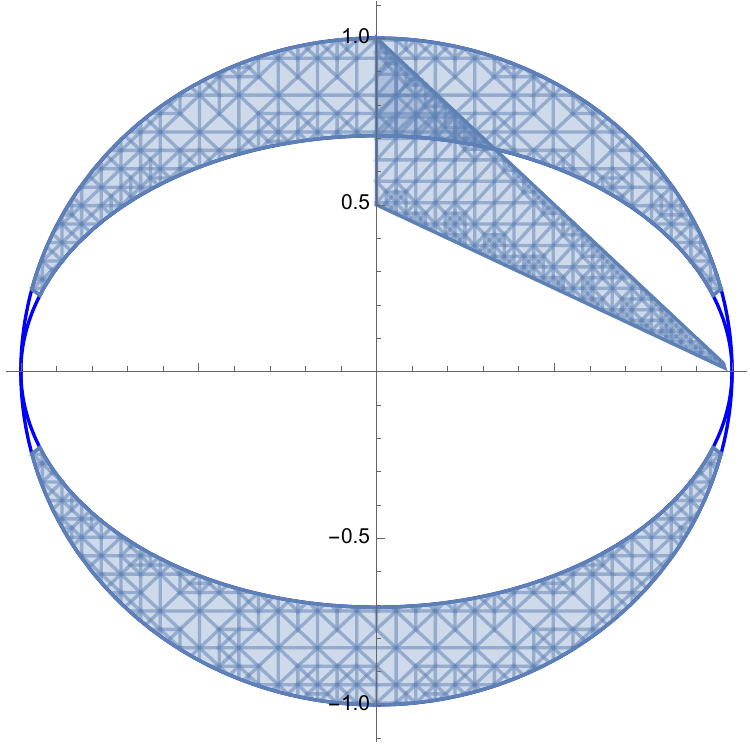}
\end{minipage}\hfill
\begin{minipage}[b]{0.52\textwidth}
\centering
\includegraphics[width=0.75\textwidth]{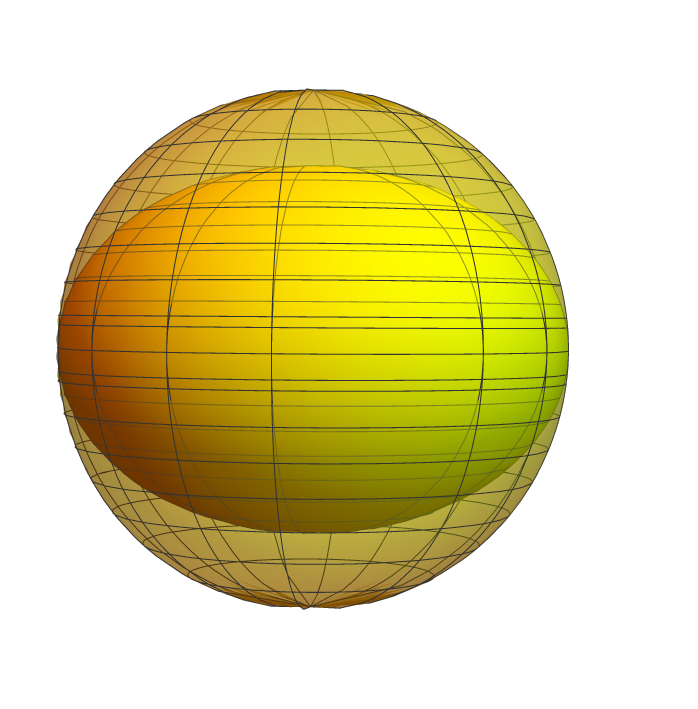}
\end{minipage}
\hspace*{\fill}
\caption{Domain of two ellipsoids with $\fa=1$, $\fb = \f12$} \label{fig:ellipsoid}
\end{figure}

Let $\sw_{\a,\b,\g,\t}$ be the weight function defined in \eqref{eq:w-4parameters}. For $(s,t) \in \sqrt{\Omega_{\fa,\fb}}$, let
\begin{align*}
  \sw^{\fa,\fb}_{\a,\b,\g,\t}(s,t) \, & = s^{\a} \big(\fa (1-s)-t\big)^\b \big( t-\fb(1-s) \big)^\g \big(t-\fb+ \fa s\big)^{\t} \\
        & = (\fa-\fb)^{\b+\g+\t} \sw_{\a,\b,\g,\t}\left(s, \frac{t-\fb +\fa s}{\fa-\fb}\right).
\end{align*}
The corresponding weight function on the rotational solid $\XX_{\fa,\fb}^{d+1}$ is given by
\begin{align*}
 \Wb_{\a,\b,\g,\t}^{\fa,\fb}(x,t) \, & = \|x\|^{2 \a} \big(\fa(1-\|x\|^2)- t^2 \big)^\b  
        \big(t^2- \fb (1-\|x\|^2)\big)^\g \big(t^2-\fb + \fa \|x\|^2\big)^{\t} \\
   & = (\fa-\fb)^{\b+\g+\t} \Wb_{\a,\b,\g,\t} \left(\|x\|^2, \frac{t^2-\fb+\fa \|x\|^2}{\fa-\fb} \right).
\end{align*}
where $\Wb_{\a,\b,\g,\t}$ is defined in \eqref{eq:W-4paraHyp}. For $0< \fb < \fa$, orthogonal bases on the domain 
$\XX^{d+1}_{\fa,\fb}$ can be derived from the orthogonal bases on the domain $\XX^{d+1}_{1, \fb}$, but only for
the space of polynomials even in the $t$-variable. 

\begin{prop}
Let $\b,\g > -1$, $\a > -\f d2$, and $\a+\g+\t \ge -\f d2$. Let $\{\Qb_{j,k,\ell}^n \}$ be the orthogonal basis of 
the space $\CV_n^\sE(\Wb_{\a,\b,\g,\t}, \XX^{d+1})$ in Proposition \ref{prop:OPDouble_Va}. Then the polynomials 
$$
    \Qb_{j,k,\ell}^{(\fa,\fb), n} (x,t) =  \Qb_{j,k,\ell}^n \left(x,  \sqrt{\frac{t^2-\fb+\fa \|x\|^2}{\fa-\fb}}\right), 
    \quad j \le k/2, \,\, 0 \le k \le n
$$
and $1 \le \ell \le \dim \CH_{k-2j}^d$, consist of an orthogonal basis for 
$\CV_n^\sE\big(\Wb^{\fa,\fb}_{\a,\b,\g,\t}, \XX^{d+1}_{\fa,\fb}\big)$. 
\end{prop}

If $\fb =0$ and $\fa = 1$, then the domain $\XX_{1,0}^{d+1} = \BB^{d+1}$ is the unit ball in $\RR^{d+1}$ and the 
weight function is 
$$
  \Wb_{\a,\b,\g,\t}^{1,0}(x,t) = \|x\|^{2 \a} \big(1-\|x\|^2 - t^2 \big)^\b  
          |t|^{2 \g} \big(t^2-\|x\|^2\big)^{\t},
$$
which becomes the classical OPs on the unit ball when $\a = \g = \t =0$, as given in 
\eqref{eq:basisBd}, and a basis in the vein of \eqref{eq:basisBd} is also known if $\a \ne 0$ or $\g \ne 0$. 
However, the basis is new if $\t \ne 0$. The spectral differential operator and the addition formula
are known for the classical OPs. 
For $\fb = 0$ and $\t =\f12$, both these properties also hold for the space 
$\CV_n^\sE(\Wb_{0,\b,\g,\f12}^{1,\fb}, \BB^{d+1})$, as stated below, which are new. 

\begin{thm}
Let $0 \le \fb < \fa$. Then every $u \in \CV_n^\sE \big(\Wb^{\fa,\fb}_{0,\b,\g,\f12}, \XX_{\fa,\fb}^{d+1}\big)$ satisfies
an analog of \eqref{eq:CD1b} with $\CD_{x,t}^{1,\fb}$ replaced by 
$$
 \CD^{\fa,\fb}_{x,t}(\partial_x, \partial_t) = \CD_{x, z}(\partial_{x}, \partial_z), \qquad z = \sqrt{\frac{t^2-\fb+ \fa\|x\|^2}{\fa-\fb}}. 
$$
Moreover, the addition formula \eqref{eq:PEadd} holds for the kernel $\Pb_n^\sE  \big  (\Wb^{\fa,\fb}_{0,\b,\g,\f12}; (x,t),(y,s) \big)$ 
with $\zeta(x,t,y,s;u,v)$ replaced by 
\begin{align*}
  \zeta^{\fa,\fb} (x,t, y,s; u, v) = \,&
      \Big(\la x,y\ra  + \frac{u}{\fa-\fb} \sqrt{t^2-\fb (1-\|x\|^2)}\sqrt{s^2-\fb(1-\|y\|^2)} \Big) \mathrm{sign}(st) \\
    & + \frac{v}{\fa - \fb} \sqrt{\fa - t^2-\fa \|x\|^2}\sqrt{\fa-s^2 -\fa \|y\|^2}.  
\end{align*}
\end{thm}

\subsection{Applications}
A framework is developed in \cite{X21} for approximation theory and computational harmonic analysis on the space of 
homogeneous type $(\Omega, \sw, \sd)$, where $\Omega$ is a domain in a Euclidean space, $\sw$ is a doubling weight 
for the metric (distance) $\sd$ on $\Omega$. It is based on highly localized kernels derived from OPs for $\sw$ on $\Omega$,
which are established with the help of an addition formula, and uses the spectral operator to relate approximation by 
polynomials and smoothness of function. The framework applies to the Gegenbauer polynomials on the hyperboloid 
in \cite{X23a} and leads to a characterization of best approximation by polynomials and localized tight frame, among 
several other results.

Since the new domains and weight functions in this section are related to double cone and hyperboloid by a change of 
variable, so much so that both spectral operator and addition formula are preserved, it follows readily that the framework
in \cite{X21} is applicable and the aforementioned results hold on the new domains discussed in this section as well.

\end{document}